\documentclass[11pt, centertags,oneside]{amsart}

\usepackage{amsmath,amstext,amsthm,amscd,typearea,hyperref,stmaryrd}
\usepackage{amssymb}
\usepackage{a4wide}
\usepackage[mathscr]{eucal}
\usepackage{mathrsfs}
\usepackage{typearea}
\usepackage{charter}
\usepackage{pdfsync}
\usepackage{MnSymbol}
\usepackage{framed}
\usepackage[a4paper,width=16.2cm,top=3cm,bottom=3cm]{geometry}

\numberwithin{equation}{section}

\usepackage{xcolor}
\usepackage{tikz}

\newtheorem{theorem}{Theorem}[section]
\newtheorem{definition}[theorem]{Definition}
\newtheorem{proposition}[theorem]{Proposition}
\newtheorem{corollary}[theorem]{Corollary}
\newtheorem{lemma}[theorem]{Lemma}
\newtheorem{remark}[theorem]{Remark}

\newtheorem{example}[theorem]{Example}

\newcommand{\cali}[1]{\mathscr{#1}}

\newcommand{\supp}{{\rm supp}}

\newcommand{\dist}{{\rm dist}}

\newcommand{\loc}{{loc}}

\newcommand{\reg}{{\rm reg}}

\newcommand{\Cc}{\cali{C}}
\newcommand{\Dc}{\cali{D}}

\newcommand{\Fc}{\cali{F}}
\newcommand{\Gc}{\cali{G}}

\newcommand{\Kc}{\cali{K}}

\newcommand{\Uc}{\cali{U}}

\newcommand{\Sc}{\cali{S}}
\newcommand{\Tc}{\cali{T}}

\newcommand{\C}{\mathbb{C}}

\newcommand{\N}{\mathbb{N}}

\newcommand{\R}{\mathbb{R}}

\renewcommand{\P}{\mathbb{P}}

\newcommand{\Xf}{{\mathfrak{X}}}
\newcommand{\Xfh}{{\widehat{\mathfrak{X}}}}

\newcommand{\euc}{\mathrm{euc}}

\newcommand{\Uf}{{\mathfrak{U}}}

\newcommand{\Kmc}{{\mathcal{K}}}
\newcommand{\Tmc}{{\mathcal{T}}}

\newcommand{\Bmc}{{\mathcal{B}}}
\newcommand{\Cmc}{{\mathcal{C}}}

\newcommand{\bl}{\mathrm{bl}}

\title[Continuous local potential functionals]{Continuous local potential functionals and the Dinh-Sibony product}

\author{Taeyong Ahn}
\address{(Ahn) Department of Mathematics Education, Inha University, 100 Inha-ro, Michuhol-gu, Incheon 22212, Republic of Korea}%
\email{t.ahn@inha.ac.kr}

\date{\today}


\begin{document}
\begin{abstract}
In this article, we generalize the notion of continuous superpotentials on compact K\"ahler manifolds to arbitrary complex manifolds in terms of local potential functionals and study related properties. In particular, we study the associativity of the Dinh-Sibony product and also a sufficient condition for the continuity of the Dinh-Sibony product.
\end{abstract}
\maketitle

\section{Introduction}

The intersection of positive closed currents has been one of the central problems in pluripotential theory. The bidegree $(1, 1)$-case has been studied using classical pluripotential theory (e.g., \cite{BT}, \cite{FS}, \cite{BEGZ}). For the general bidegree cases, two remarkable theories have been used: superpotentials (\cite{DS09}) and tangent currents (\cite{DS18}).\medskip

Superpotentials provided a reasonable theory in the case of complex projective spaces in \cite{DS09}. Superpotentials are applicable to compact K\"ahler manifolds but, due to the lack of positivity in certain senses, superpotentials have not given as successful a theory for compact K\"ahler manifolds as for complex projective spaces. (\cite{DS10}) Then, with more emphasis on positivity, the theory of tangent currents and the notion of the Dinh-Sibony product were introduced in \cite{DS18}. In \cite{Nguyen2501}, \cite{Ahn25}, reasonable conditions for the Dinh-Sibony product were introduced on compact K\"ahler manifolds and in \cite{Ahn25} on general complex manifolds. In \cite{VuMich}, \cite{DNV}, the theories of superpotentials and tangent currents were shown to give the same intersection in the case of continuous superpotentials. Also, for applications of continuous superpotentials in complex dynamics, see \cite{DS10} for instance. See also \cite{VU21-1}, \cite{VU21-2}, \cite{Vu}, \cite{NguyenPositive}, \cite{NT}, \cite{LZ} for other developments of tangent currents.\medskip

In this work, we further look into properties of tangent currents, in particular, the associativity properties of the Dinh-Sibony product. By the nature of the definition of tangent currents, unlike the binary intersection operation, for given three positive closed currents $S_1$, $S_2$, $S_3$, we can directly define the intersection $\left(S_1\wedge S_2\wedge S_3\right)_{DS}$ of $S_1$, $S_2$, $S_3$. Here, the subscript means the Dinh-Sibony product (see Definition \ref{def:DSproduct}). However, it is not obvious that $\left(S_1\wedge S_2\wedge S_3\right)_{DS}$ equals the iterated intersections $((S_1\wedge S_2)_{DS}\wedge S_3)_{DS}$ and $(S_1\wedge (S_2\wedge S_3)_{DS})_{DS}$ when they are defined.\medskip

The motivation of this work is to prove this consistency under some regularity assumptions on currents. To this end, we extend the notion of continuous superpotentials in the context of tangent currents and prove the associativity property of the Dinh-Sibony product under the assumption of continuity. The main theorem is as follows:\medskip

Let $X$ be a complex manifold of dimension $n$. For $i=1, 2, 3$, let $S_i$ be a positive closed $(s_i, s_i)$-current, where $1\le s_1+s_2+s_3\le n$. Let $\left(\left(U_j^i, \xi_j^i\right)\right)_{i=0, \ldots, n, j\in J}$ be localizing data (see Subsection \ref{subsec:local_pot_supftns}).  For each $j\in J$, we consider 
$K_j=\{R\in\Cc_{n-p+1}(U_1^1; U_j^0): \|R\|_{*, j}\le 1\}$ (see Subsection \ref{subsec:positive_currents}). For local potential functionals, refer to \ref{subsec:local_pot_supftns} and for continuous ones, Section \ref{sec:conti}. 
\begin{theorem}\label{thm:associativity}
	Suppose that $S_1$ admits continuous local potential functionals on $\left(K_j\right)_{j\in J}$ and that $S_2$ and $S_3$ satisfy Condition (I) on $U_j^0$ for $j\in J$. Then, the Dinh-Sibony product $\left((S_1\wedge S_2)_{DS}\wedge S_3\right)_{DS}$ of $\big(S_1\wedge S_2\big)_{DS}$ and $S_3$, and the Dinh-Sibony product $(S_1\wedge S_2\wedge S_3)_{DS}$ of $S_1$, $S_2$ and $S_3$ are well defined for each $j\in J$, and we have
	\begin{align*}
		\left(S_1\wedge (S_2\wedge S_3)_{DS}\right)_{DS}=\left((S_1\wedge S_2)_{DS}\wedge S_3\right)_{DS}=(S_1\wedge S_2\wedge S_3)_{DS}\textrm{ on }X.
	\end{align*} 
\end{theorem}
For Condition (I), see Definition \ref{def:conditionI}. Intuitively, not requiring closedness outside $\overline{U_j^1}$ for currents in $\Cc_{n-p+1}(U_j^1; U_j^0)$ is for the localization. Local potential functionals were introduced in \cite{Ahn25} and they are functions defined on each $K_j$. The definition actually resembles superpotentials. However, compared with superpotentials, we cannot use cohomological arguments. As will be seen in Proposition \ref{prop:char_conti} and Theorem \ref{thm:cptK}, globally, it plays a similar role to superpotentials.\medskip 

Also, following \cite{DNV}, we characterize the continuity of local potential functions in terms of integration near the diagonal submanifold as in Proposition \ref{prop:char_conti}. Indeed, this characterization plays a crucial role in this work.\medskip

When $X$ is a compact K\"ahler manifold, we can describe the continuity of superpotentials in terms of local potential functionals as below:
\begin{theorem}\label{thm:cptK}
	Let $(X, \omega_X)$ be a compact K\"ahler manifold of dimension $n$.  Let $S\in\Cc_p(X)$. Then, $S$ admits continuous superpotentials if and only if $S$ admits continuous local potential functionals on $K$, where $K$ is the closure of the set of smooth positive closed $(n-p+1, n-p+1)$-currents of unit mass.
\end{theorem}

In this sense, the notion of continuous local potential functionals generalizes that of the continuous superpotentials. Furthermore, in the case of compact K\"ahler manifolds, we improve the associativity property as below.
\begin{theorem}\label{thm:associativity_cptK}
	Let $(X, \omega_X)$ be a compact K\"ahler manifold of dimension $n$. Let $S_1$, $S_2$ and $S_3$ be positive closed currents of bidegree $(s_1, s_1)$, $(s_2, s_2)$ and $(s_3, s_3)$, respectively. Suppose that $S_1$ admits continuous superpotential and $S_2$ and $S_3$ satisfy Condition (I). Then, the Dinh-Sibony product of $\big(S_1\wedge S_2\big)_{DS}$ and $S_3$, and the Dinh-Sibony product of $S_1$, $S_2$ and $S_3$ are well defined for each $j\in J$, and we have
	\begin{align*}
		\left(S_1\wedge (S_2\wedge S_3)_{DS}\right)_{DS}=\left((S_1\wedge S_2)_{DS}\wedge S_3\right)_{DS}=(S_1\wedge S_2\wedge S_3)_{DS}\textrm{ on }X.
	\end{align*} 
\end{theorem}

Besides the associativity, as in \cite{DNV}, domination priciple works for continuous local potential functionals (see Theorem \ref{thm:domination}). Not directly related to associativity, with the same idea, one can find a sufficient condition for the continuity of the Dinh-Sibony product as in Theorem \ref{thm:coni_DS_prod}.\medskip

The main technical difficulty of this work comes from the presence of the boundary and the comparison of local quantities and global quantities. The superpotentials are functions defined on positive closed currents. Positive closed currents are a global object and keeping closedness when localizing is not an easy task. However, in \cite{DNV}, the continuity of superpotentials was described in terms of integration near the exceptional divisor. Due to the positivity of currents under consideration, integration works very will with localization. So, as explained in the above, we exploit the description of the continuous superpotentials in \cite{DNV} and generalize this property. In connection with continuity, uniform convergence for continuous local potential functionals is also important. Another idea is that instead of positive closed currents, we work with positive currents on $U_j^0$ which are closed in a relatively compact open subset $U_j^1$. Intuitively, this is for the localization of the form $\chi S$ and for regularizations of positive closed currents, where $S$ is a positive closed current and $\chi$ is a cut-off function. For a technical local regularization, see Lemma \ref{lem:DSHness}. Note that for this classes of currents, compactness property still holds as in Proposition \ref{prop:compactness}.\medskip

The paper is organized as follows. In Section \ref{sec:prelimimaries}, preliminaries on tangent currents, the Dinh-Sibony product and local potential functionals are discussed. In Section \ref{sec:conti}, we introduce continuous local potential functions. In Section \ref{sec:associativity}, associativity of the Dinh-Sibony product under the continuity of local potential functionals is proved. Theorem \ref{thm:associativity} is proved. In Section \ref{sec:cptK}, the compact K\"ahler manifold case is discussed. Theorems \ref{thm:cptK} and \ref{thm:associativity_cptK} are proved. In Section \ref{sec:conti_DS}, the continuity of the Dinh-Sibony product is briefly discussed.\medskip

\noindent {\bf Notations. } Let $\chi: \R\to \R_{\ge 0}$ be a fixed convex increasing function such that $\chi(t)=0$ for $t\le -1$ and $\chi(t) = t$ for $t\ge 1$. For a set $W$ in a metric space, $W_\varepsilon$ denotes the $\varepsilon$-neighborhood of $W$ with respect to the given metric. We use $\omega_\euc$ for the standard Euclidean K\"ahler form on $\C^n$. For a complex manifold $\Xf$, we denote by $\Cc_p(\Xf)$ the set of positive closed currents and $\widetilde{\Cc_p(\Xf)}$ its subset of smooth ones. Let $\Dc_p(\Xf)$ denote the real vector space spanned by $\Cc_p(\Xf)$.
\medskip

\noindent
\textbf{Acknowledgments.} This work is dedicated to Professor Kang-Tae Kim, in celebration of his outstanding contributions. The research of the author was supported in part by the National Research Foundation of Korea (NRF) grant funded by the Korea government (MSIT) (No. RS-2023-00250685).

\section{Preliminaries}\label{sec:prelimimaries}

\subsection{Tangent currents} The tangent current was introduced in \cite{DS18} by Dinh-Sibony and further studied in \cite{Vu}, \cite{NguyenPositive}, \cite{Nguyen2501} and \cite{Ahn25}.\medskip

Let $\Xf$ be a complex manifold of dimension $N$ and $V\subset \Xf$ a complex submanifold of codimension $n$. Let $T\in\Cc_p(\Xf)$ be such that $T$ has no mass on $V$. Let $E$ be the normal bundle of $V$ in $\Xf$ and $\overline E:=\P(E\oplus \C)$ the projective compactification of $E$. Let $\pi_V:\overline{E}\to V$ denote the canonical projection. The hypersurface at infinity $H_\infty:=\overline E\setminus E$ of $\overline E$ is isomorphic to $\P(E)$ as a fiber bundle over $V$. We have another canonical projection $\pi_\infty:\overline{E}\setminus V\to H_\infty$.
\begin{definition}
	Let $U$ be an open subset of $\Xf$ with $U\cap V\ne \emptyset$. A (local) holomorphic admissible map is a biholomorphism $\tau$ from $U$ to an open neighborhood of $U\cap V$ in $E$, which is the identity on $U\cap V$, and the restriction of whose differential $d\tau$ to ${U\cap V}$ is the identity. 
\end{definition}
Here, ``local'' means that it may not be defined in a global neighborhood of the entire submanifold $V$ in $X$ but is defined in the neighborhood $U$ of $U\cap V$. The admissible maps in \cite{DS18} are global but may not be holomorphic.\medskip 

For $\lambda\in\C^*$, let $A_\lambda:E\to E$ be the multiplication by $\lambda$ on fibers of $E$. It can be extended to $\overline{E}$. The following definition of tangent current is the version in \cite{Vu}. See also \cite{DS18}.
\begin{definition}[Definition 2.1 in \cite{Vu}]\label{def:tangent_current}
	A tangent current $T_\infty$ of $T$ along $V$ is a positive closed current on $E$ such that there exists a sequence $(\lambda_k)_{k\in\N}\subset \C^*$ converging to $\infty$ and a collection of holomorphic admissible maps $\tau_i:U_i\to E$ for $i\in I$ satisfying the following two properties:
	\begin{align*}
		&(i) V\subset \bigcup_{i\in I} U_i\\
		&(ii) T_\infty :=\lim_{k\to\infty} \left(A_{\lambda_k}\right)_*(\tau_i)_*T \,\,\textrm{ on }\,\pi_V^{-1}(U_i\cap V)\setminus H_\infty\,\,\textrm{ for every }\,i\in I.
	\end{align*}
	A tangent current $T_\infty$ trivially extends to $\overline{E}$. We still denote it by $T_\infty$.
	For an open subset $V_0$ of $V$, the horizontal dimension (or the $h$-dimension for short) of $T_\infty$ over $V_0$ is the largest integer $h_T$ such that $T_\infty\wedge \pi_V^*(\omega_V^{h_T})\ne 0$ on $\pi_V^{-1}(V_0)$, where $\omega_V$ is a K\"ahler form on $V$. The $h$-dimension of $T_\infty$ is its $h$-dimension over $V$.
\end{definition}

\begin{remark}
	If $h_T$ is the $h$-dimension of $T_\infty$, then we have $\max\{N-p-n, 0\}\le h_T\le \min\{N-p, N-n\}$. The $h$-dimension $h_T$ of $T_\infty$ is said to be minimal, if $h_T=\max\{N-p-n, 0\}$.
\end{remark}
\begin{remark}\label{rmk:independence_of_admissible_map}
	Thanks to \cite[Proposition 2.5]{KV}, tangent currents are independent of the choice of holomorphic admissible maps $(\tau_i)_{i\in I}$. In particular, they are well-defined on any intersection $\pi_V^{-1}(U_i\cap V)\cap \pi_V^{-1}(U_j\cap V)$ whenever the limits exist.
\end{remark}

\begin{definition}[Definition 3.6 in \cite{DS18}]\label{def:shadow}
	Let $T_\infty$ be a tangent current of $T$ along $V$ with its $h$-dimension $h_T$ over an open subset $V_0$ in $V$. Let $\Omega$ be a smooth closed $\left(N-p-h_T, N-p-h_T\right)$-form on $\pi_V^{-1}(V_0)$ whose restriction to each fiber of $\pi_V$ is cohomologous to a linear subspace in this fiber. The shadow of $T_\infty$ on $V_0$ is the positive closed $(N-n-h_T, N-n-h_T)$-current $T_\infty^h:=\left(\pi_V\right)_*\left(T_\infty\wedge\Omega\right)$ with support in $\pi_V\left(\supp\left(T_\infty\right)\right)\cap V_0$. The shadow of $T_\infty$ is defined to be the shadow on $V$.
\end{definition}

\begin{remark}\label{rmk:indep_shadow}
	\cite[Proposition 3.5]{DS18} says that the shadow $T_\infty^h$ is independent of the choice of $\Omega$. 
\end{remark}

In this work, we focus on the Dinh-Sibony product of positive closed currents rather than tangent currents of general positive closed currents.\medskip
\begin{definition}[Definition 5.9 in \cite{DS18}]\label{def:DSproduct} Let $X$ be a complex manifold of dimension $n$. Let $S_i\in \Cc_{s_i}(X)$ for $i=1, \ldots, k$, where $1\le s:=s_1+\cdots+s_k\le n$. Consider a positive closed current $\pi_1^*S_1\wedge\cdots\wedge\pi_k^*S_k$ on $X^k$, where $\pi_i:X^k\to X$ denotes the canonical projection onto the $i$-th factor for $i=1, \ldots, k$. Let $\Delta$ denote the diagonal submanifold in $X^k$. Assume that there is a unique tangent current $\big(\pi_1^*S_1\wedge\cdots\wedge\pi_k^*S_k\big)_\infty$ of $\pi_1^*S_1\wedge\cdots\wedge\pi_k^*S_k$ along $\Delta$ and that its $h$-dimension is minimal. The Dinh-Sibony product $\left(S_1\wedge \cdots \wedge S_k\right)_{DS}$ of $S_1, \ldots, S_k$ is defined to be the shadow $\left(\pi_1^*S_1\wedge\cdots\wedge\pi_k^*S_k\right)^h_\infty$ of the tangent current $\left(\pi_1^*S_1\wedge\cdots\wedge\pi_k^*S_k\right)_\infty$.
\end{definition}

\subsection{Positive currents}\label{subsec:positive_currents} We consider the following positive currents in order to handle the localization of positive closed currents by allowing non-closedness near the boundary.\medskip

Let $X$ be a complex manifold of dimension $n$. Let $\left(U_j^i\right)_{j\in J}$ for $i=0, \ldots, n$ be open coverings of $X$ such that each $U_j^i$ is biholomorphic to a bounded simply connected domain in $\C^n$ with smooth boundary and that $\overline{U_j^i}\subset U_j^{i-1}$ for each $j\in J$ and for each $i=1, \ldots, n$. We use the same coordinate chart for $U_j^0, \ldots, U_j^{n}$. Let $x$ denote the coordinates for $U_j^i$, where $i=0, 1, \ldots, n$.\medskip

Let $j\in J$ be fixed throughout this and next subsections unless stated otherwise as we can apply the following to each and every $j\in J$. We first specify the domain of local potential functionals.\medskip

Let $\Cc_p(U_j^1; U_j^0)$ be the set of positive $(p, p)$-currents of bounded mass on $U_j^0$, which are closed in $U_j^1$. Let $\widetilde{\Cc}_p(U_j^1; U_j^0)$ denote the subset of smooth currents in $\Cc_p(U_j^1; U_j^0)$. From the standard regularization by convolution, we see that $\widetilde{\Cc}_p(U_j^1; U_j^0)$ is dense in $\Cc_p(U_j^1; U_j^0)$. Let $\Dc_p(U_j^1; U_j^0)$ and $\widetilde{\Dc}_p(U_j^1; U_j^0)$ be the real vector spaces generated by $\Cc_p(U_j^1; U_j^0)$ and $\widetilde{\Cc}_p(U_j^1; U_j^0)$, respectively.\medskip 

From the bounded mass condition, every current in $\Cc_p(U_j^1; U_j^0)$ can be trivially extended to a neighborhood of $\overline{U_j^0}$. Since $\overline{U_j^0}$ is compact, the mass norm $\|S\|_j$ for $S\in \Cc_p(U_j^1; U_j^0)$ can be replaced by
\begin{align*}
	\|S\|_j:=\int_{U_j^0} S\wedge \omega_\euc^{n-p}.
\end{align*}
For $S\in \Dc_p(U_j^1; U_j^0)$, we define the $*$-norm $\|S\|_{*, j}$ on $U_j^0$ to be
\begin{align*}
	\|S\|_{*, j}:=\inf\left(\|S_+\|+\|S_-\|\right),
\end{align*}
where $\inf$ is taken over all $S_\pm\in\Cc_p(U_j^1; U_j^0)$ such that $S=S_+-S_-$.\medskip

We consider the space $\Dc_p(U_j^1; U_j^0)$ with the following topology: a sequence $\left(S_k\right)_{k\in \N}$ converges to $S$ in $\Dc_p(U_j^1; U_j^0)$ if $S_k\to S$ as $k\to\infty$ in the sense of currents and there exists an $M>0$ such that $\|S_k\|_{*, j}\le M$ for all $k\in\N$. On any subset of $\Dc_p(U_j^1; U_j^0)$ with bounded $*$-norm on $U_j^0$, this topology conicides with the classical weak topology for currents. Given open coverings $\left(\left(U_j^i\right)_{j\in J}\right)_{i=0, 1, \ldots, n}$, from a global perspective, we say that a sequence $(S_k)_{k\in\N}$ of positive closed currents on $X$ converges to a positive closed current $S$ on $X$ with respect to the local $*$-topology if for each $j\in J$, there exists $M_j>0$ such that $\|S_k\|_{*, j}\le M_j$ for $k\in \N$ and if $S_k\to S$ in the sense of currents.\medskip


\begin{proposition}\label{prop:compactness}
	The sets $\Bmc_{p, j}:=\{S\in\Dc_p(U_j^1; U_j^0): \|S\|_{*,j}\le 1\}$ and $\Bmc^+_{p, j}:=\{S\in\Cc_p(U_j^1; U_j^0): \|S\|_{*,j}\le 1\}$ are compact with respect to the $*$-topology on $U_j^0$.
\end{proposition}
\begin{proof}
	The space $\Bmc_{p, j}:=\{S\in\Dc_p(U_j^1; U_j^0): \|S\|_{*,j}\le 1\}$ is metrizable. Let $\left(S_k\right)_{k\in \N}$ be a sequence in $\Bmc_{j, p}$. We show that there exists a convergence subsequence. For each $k\in\N$, we can find two positive closed currents $S_k^\pm$ such that $S_k=S_k^+-S_k^-$ and $\|S_k^+\|_j+\|S_k^-\|_j\le 1+\frac{1}{k}$. Hence, by applying Alaoglu's theorem to $\left(S_k^+\right)_{k\in\N}$ and passing to a convergent subsequence, we may assume that $\left(S_k^+\right)_{k\in\N}$ is convergent. Again applying the same to $\left(S_k^-\right)_{k\in\N}$, we may assume that both sequences $\left(S_k^+\right)_{k\in\N}$ and $\left(S_k^-\right)_{k\in\N}$ are convergent. Let $S^\pm$ denote their respective limit currents and we clearly see $S_k\to S:=S^+-S^-$. The positivity of $S$ and the closedness of $S^\pm$ in $U_j^1$ are clear. Let $\left(\chi_l\right)_{l\in\N}$ be the set of smooth functions with compact support such that $\chi_l\le \chi_{l+1}$ and $\lim_{l\to\infty}\chi_l =\mathbf{1}_{U_j^0}$. Then, we have
	\begin{align*}
		\int \chi_l(S^++S^-)\wedge\omega_\euc^{n-p}=\lim_{k\to\infty}\int \chi_l(S_k^++S_k^-)\wedge\omega_\euc^{n-p}\le 1
	\end{align*}
	By the regularity of Radon measures, letting $l\to\infty$, we see that $\|S\|_{*, j}\le 1$.\medskip
	
	The same argument applies for the compactness of $\Bmc_{p, j}^+$.
\end{proof}
\bigskip

\subsection{Local potential functionals}\label{subsec:local_pot_supftns} 
In this work, we adopt the local approach to the Dinh-Sibony product as in \cite{Ahn25}, using local potential functionals. For a technical reason, we consider refined but slightly more singular ones. However, for our purpose, they are good enough.\medskip


We will call the collection $\left(\left(U_j^i, \xi_j^i\right)\right)_{i=0, 1, \ldots, n, j\in J}$, which will be defined below, localizing data of $X$. For each $i=1, \ldots, n$, let $\xi_j^i:U_j^{i-1}\to [0, 1]$ be a smooth function with compact support  such that $\{\xi_j^i\equiv 1\}$ contains a neighborhood of $\overline{U_j^i}$. We take $\xi_j^0$ to be just a cut-off function of $\overline{U_j^0}$. We use the coordinates $(x, y)\in \C^n\times\C^n$ and $\pi_k: \C^n\times \C^n\to \C^n$ the canonical projection onto the $k$-th factor for $k=1, 2$ so that $\pi_1(x, y)=x$ and $\pi_2(x, y)=y$. We also denote by $\Delta$ the diagonal submanifold of $\C^n\times \C^n$. Let $\Uf_j^i:=U_j^i\times U_j^i$.  Let $\chi_j^i=\left(\pi_1^*\xi_j^i\right)\left(\pi_2^*\xi_j^i\right)$ a smooth function with compact support in $\Uf_j^{i-1}$ such that $\{\chi_j^i\equiv 1\}$ contains a neighborhood of $\overline{\Uf_j^i}$.\medskip 

Let $u=\log |x-y|$ and $\omega=\pi_1^*\omega_\euc+\pi_2^*\omega_\euc=dd^c|x|^2 + dd^c|y|^2$. For a given $\theta\in\C^*$ with $|\theta|\ll 1$, we consider two approximations of $u$ defined by 
\begin{align*}
	u_\theta^\Kmc:=\chi\left(u-\log |\theta|\right)+\log|\theta|\quad \textrm{ and }\quad u_\theta^\Tmc:=\frac{1}{2}\log\left(\left|x-y\right|^2+|\theta|^2\right),
\end{align*}
both of which decreasingly converge to $u$ as $|\theta|\to 0$. For $k=1, \ldots, n$, we will also use
\begin{align*}
	\Kmc_\theta^k:=dd^cu_\theta^\Kmc\wedge (dd^c u)^{k-1},
\end{align*}
which is a smooth positive closed $(i, i)$-form on $\C^n\times \C^n$. 

\begin{definition}\label{def:prod_C}
	Let $T$ be a positive $(p, p)$-current on a domain in $\C^n\times\C^n$. Let $k\in \{1, \ldots, n\}$. Then, we define the current
	\begin{align*}
		\left\langle (dd^cu)^k \wedge T \right\rangle_C:=\lim_{\theta\to 0}\Kmc_\theta^k\wedge T,
	\end{align*} 
	provided that the limit exists in the sense of currents on $\C^n\times \C^n$.
\end{definition}

The following lemma is straightforward. See also \cite[Proposition 3.17]{Ahn25}.
\begin{lemma}\label{lem:induction}
	Let $T$ be a positive closed $(p, p)$-current on a domain in $\C^n\times\C^n$. Let $k\in\{1, \ldots, n\}$. Suppose that the current $\left\langle (dd^cu)^{k-1}\wedge T\right\rangle_C$ is well defined and suppose that $u$ is locally integrable with respect to the measure $\left\langle (dd^cu)^{k-1}\wedge T\right\rangle_C\wedge \omega^{2n-p-k+1}$. Then, $\left\langle (dd^cu)^k\wedge T\right\rangle_C$ is well defined.
\end{lemma}


Let $K_j\subseteq \Cc_{n-p+1}(U_j^1; U_j^0)$ be a compact subset with uniformly bounded $*$-norm on $U_j^0$ such that $\widetilde{K}_j$ is dense in $K_j\subseteq \Cc_{n-p+1}(U_j^1; U_j^0)$, where $\widetilde{K}_j:=K_j\cap \widetilde{\Cc}_{n-p+1}(U_j^1; U_j^0)$.\medskip

Let $S\in \Cc_p(X)$. We consider the following functionals. For $R\in K_j$, we define
\begin{align*}
	\Fc_{S, j, \theta}^1(R):=\int_{\Uf^0_j} \chi_j^1u_\theta^\Kmc  \pi_1^*S\wedge \pi_2^*R\wedge \omega^{n-1}.
\end{align*}
This family of functions is decreasing as $|\theta|$ decreases to $0$. We allow $-\infty$ as value and we take the limit of the family as $|\theta|\to 0$. We denote by $\Fc_{S, j}^1$ the limit function and by $\Cmc_{S, j}^1:=\{R\in\Cc_{n-s+1}(U_j^1; U_j^0): \Fc_{S, j}^1(R)>-\infty\}$ its domain of finite values. 
According to Lemma \ref{lem:induction}, for $R\in \Cmc_{S, j}^1$, the product $\left\langle\pi_1^*S\wedge \pi_2^*R\wedge dd^cu \right\rangle_C$ is well defined on $\Uf_j^1$. Then, for $R\in \Cmc_{S, j}^1$ (not on the whole space of $K_j$), we define
\begin{align*}
	\Fc_{S, j, \theta}^2(R):=\int_{\Uf_j^1} \chi_j^2u_\theta^\Kmc  \left\langle\pi_1^*S\wedge \pi_2^*R\wedge dd^cu \right\rangle_C\wedge \omega^{n-2},
\end{align*}
which decreases as $|\theta|$ decreases. As previously, we take the limit and denote by $\Fc_{S, j}^2$ the limit function on $\Cmc_{S, j}^1$; we define its associated set $\Cmc_{S, j}^2$. Inductively, we define $\Fc_{S, j, \theta}^i$, $\Fc_{S, j}^i$ and $\Cmc_{S, j}^i$ for $i=1, \ldots, n$. For $i=1, \ldots, n$ and for $R\in \Cmc_{S, j}^i$, we have
{
	\begin{align*}
		\Fc_{S, j, \theta}^i(R)&=\int_{\Uf_j^{i-1}} \chi_j^iu_\theta^\Kmc  \left\langle\pi_1^*S\wedge \pi_2^*R\wedge \left(dd^cu\right)^{i-1} \right\rangle_C\wedge \omega^{n-i}\\
		&=\int_{\Uf_j^{i-1}\setminus \Delta} \chi_j^iu_\theta^\Kmc\left(dd^cu\right)^{i-1}\wedge\pi_1^*S\wedge \pi_2^*R\wedge \omega^{n-i}
\end{align*}}
and
{
	\begin{align*}
		\Fc_{S, j}^i(R)&=\int_{\Uf_j^{i-1}} \chi_j^iu  \left\langle\pi_1^*S\wedge \pi_2^*R\wedge \left(dd^cu\right)^{i-1} \right\rangle_C\wedge \omega^{n-i}\\
		&=\int_{\Uf_j^{i-1}\setminus \Delta} \chi_j^iu \left(dd^cu\right)^{i-1}\wedge\pi_1^*S\wedge \pi_2^*R\wedge \omega^{n-i}.
\end{align*}}

\begin{definition}
	Let $S\in\Cc_p(X)$. Let localizing data $\left(\left(U_j^i, \xi_j^i\right)\right)_{i=0, 1, \ldots, n, j\in J}$ be given. The collection $\left(\Fc_{S,j}^i\right)_{i\in\{1, \ldots, n\}, j\in J}$ as defined above is called a collection of local potential functionals associated with $\big(\big(U_j^i, \xi_j^i\big)\big)_{i=0, 1, \ldots, n, j\in J}$. Each $\Fc_{S, j}^i$ is called the $i$-th local potential functional on $U_j^i$. If the data are well understood, we will simply call it local potential functionals.
\end{definition}

As in \cite[Proposition 5.8]{Ahn25}, $\Fc^i_{S, j, \theta}$ and $\Fc^i_{S, j}$ have the following regularity property.
\begin{proposition}\label{prop:usc_local}
	Let $i\in\{1, \ldots, n\}$, $j\in J$ and $\theta\in\C^*$ with $|\theta|\ll1$. Then, $\Fc^i_{S, j, \theta}$ and $\Fc^i_{S, j}$ are upper-semicontinuous on $\Cmc_{S, j}^i$ in the following sense. Let $R\in\Cmc_{S,j}^i$ be a current and $\left(R_k\right)_{k\in\N}\subset\Cmc^i_{S,j}$ be a sequence of currents such that $R_k\to R$ as $k\to\infty$ in the sense of currents. Then, we have
	\begin{align*}
		\limsup_{k\to \infty}\Fc_{S, j, \theta}^i(R_k)\le \Fc_{S, j, \theta}^i(R)\quad\textrm{ and }\quad\limsup_{k\to \infty}\Fc_{S, j}^i(R_k)\le \Fc_{S, j}^i(R).
	\end{align*}
\end{proposition}

\section{Continuous local potential functionals}\label{sec:conti}
In this section, we introduce a generalization of positive closed currents with continuous superpotentials and characterize it in terms of functions related to the diagonal submanifold $\Delta$ as in \cite{DNV}.\medskip 

%

We are interested in the continuity of the functions $\Fc_{S, j}^i$ for $i=1, \ldots, n$ and $j\in J$. 
Observe that each $\Fc_{S, j}^i$ is well-defined on $\widetilde{\Cc}_{n-s+1}(U_j^1; U_j^0)$. 
\begin{definition}[Continuous local potential functionals.]\label{def:conti_loc_supftn}
	Let $S\in\Cc_p(X)$. Let $\left(\left(U_j^i, \xi_j^i\right)\right)_{i=0, 1, \ldots, n, j\in J}$ be localizing data. For each $j\in J$, let $K_j$ be a compact subset of $ \Cc_{n-p+1}(U_j^1; U_j^0)$ with uniformly bounded $*$-norm on $U_j^0$ such that $\widetilde{K}_j$ is dense in $K_j$, where $\widetilde{K}_j:=K_j\cap \widetilde{\Cc}_{n-p+1}(U_j^1; U_j^0)$. 
	We say that $S$ admits continuous local potential functionals on $\left(K_j\right)_{j\in J}$ if each $\Fc_{S, j}^i$ extends continuously to the entire $K_j$ with respect to the $*$-topology on $U_j^0$.\medskip
	
	When there exists a compact set $K\subset \Cc_{n-p+1}(X)$ such that $K_j=K$ for every $j\in J$, we simply say that $S$ admits continuous local potential functionals on $K$.
\end{definition}

In the rest of the section, we use fixed localizing data $\left(\left(U_j^i, \xi_j^i\right)\right)_{i=0, 1, \ldots, n, j\in J}$, $S\in \Cc_p(X)$ and compact sets $\left(K_j\right)_{j\in J}$ unless stated otherwise.\medskip

Following the characterization of continuous superpotentials as in \cite{DNV}, we introduce the following function. Let $i\in\{1, \ldots, n\}$ and $j\in J$. Let $\delta\ll 1$ be a positive real number. We define
\begin{align*}
	\nu_{S,j}^i(\delta):=\sup\int_{\Delta_\delta}-\chi_j^iu \left(dd^cu\right)^{i-1}\wedge\pi_1^*S\wedge \pi_2^*R\wedge \omega^{n-i},
\end{align*}
where the supremum is taken over $R\in\widetilde{K}_j$ such that $\|R\|_j\le 1$.

\begin{proposition}\label{prop:conti_usc_local}
	Let $i\ge 2$ and $j\in J$. Assume that $\displaystyle \lim_{\delta\to 0}\nu_{S, j}^{i-1}(\delta)= 0$. For $\theta\in\C^*$ with $|\theta|\ll1$, $\Fc^i_{S, j, \theta}$ and $\Fc^i_{S, j}$ are upper-semicontinuous on $\Cmc_{S, j}^{i-1}$ in the following sense. Let $R\in\Cmc_{S,j}^{i-1}$ be a current and $\left(R_k\right)_{k\in\N}\subset\Cmc^{i-1}_{S,j}$ be a sequence of currents such that $R_k\to R$ as $k\to\infty$ in the sense of currents. Then, we have
	\begin{align*}
		\limsup_{k\to \infty}\Fc_{S, j, \theta}^i(R_k)\le \Fc_{S, j, \theta}^i(R)\quad\textrm{ and }\quad\limsup_{k\to \infty}\Fc_{S, j}^i(R_k)\le \Fc_{S, j}^i(R).
	\end{align*}
\end{proposition}
The proof comes after Lemmas \ref{lem:def_classic_prod} and \ref{lem:conti_conti}.
\begin{lemma}\label{lem:def_classic_prod}
	Let $i\ge 2$ and $j\in J$. Suppose that we have $\displaystyle\lim_{\delta\to 0}\nu_{S, j}^{i-1}(\delta)= 0$. Then, the product $\left\langle \left(dd^cu\right)^{i-1}\wedge\pi_1^*S\wedge \pi_2^*R\right\rangle_C$ is well defined on $\Uf_j^{i-1}$ for every $R\in\Cmc_{S,j}^{i-1}$.
\end{lemma}
\begin{proof}
	Since $R\in\Cmc_{S,j}^{i-1}$, $\left\langle\left(dd^cu\right)^{i-2}\wedge\pi_1^*S\wedge \pi_2^*R\right\rangle_C$ is well defined. We show that for any sequence $(\theta_l)_{l\in\N}\subset \C^*$ converging to $0$, the sequence $\Big(\int \chi_j^{i-1}u_{\theta_l}^\Kmc \alpha\wedge\left\langle\left(dd^cu\right)^{i-2}\wedge\pi_1^*S\wedge \pi_2^*R\right\rangle_C\Big)_{l\in\N}$ is a Cauchy sequence, where $\alpha$ is a smooth positive closed $(n-i+1, n-i+1)$-form on $\Uf_j^0$.\medskip
	
	Let $\varepsilon>0$ be given. Let $M_{K_j}>0$ be such that $\sup_{R'\in K_j}\|R'\|_j\le M_{K_j}$. Let $\left(R'_m\right)_{m\in\N}$ be a sequence in $\widetilde{K}_j$ such that $\lim_{m\to\infty}R'_m=R$. Let $c>0$ be a constant such that $\alpha\le c\|\alpha\|_\infty\omega^{n-i+1}$. Then, we can find a $\delta>0$ such that $\nu_{S, j}^{i-1}(\delta)<\frac{\varepsilon}{cM_{K_j}\|\alpha\|_\infty}$. We choose $L\in\N$ such that whenever $l\ge L$, we have $|\theta_l|<\frac{\delta}{e}$. Then, for every $l', l''\ge L$ (we may assume that $|\theta_{l'}|\le|\theta_{l''}|$), we have
	\begin{align*}
		0&\ge\int \chi_j^{i-1}(u_{\theta'}^\Kmc-u_{\theta''}^\Kmc) \alpha\wedge\left\langle\left(dd^cu\right)^{i-2}\wedge\pi_1^*S\wedge \pi_2^*R\right\rangle_C\\
		&\ge \limsup_{m\to\infty}\int \chi_j^{i-1}(u_{\theta_1}^\Kmc-u_{\theta_2}^\Kmc) \alpha\wedge\left\langle\left(dd^cu\right)^{i-2}\wedge\pi_1^*S\wedge \pi_2^*R'_m\right\rangle_C\ge -cM_{K_j}\|\alpha\|_\infty\nu_{S, j}^{i-1}(\delta)>-\varepsilon.
	\end{align*}
	The second inequality is from the same argument used in the proof of the upper semi-continuity of $\Fc_{S, j}^{i-1}$ on $\Cmc_{S, j}^{i-1}$ as in Proposition \ref{prop:usc_local}. Hence, it is a Cauchy sequence as desired. It is not difficult to see that the limit is independent of the choice of the sequence $(\theta_l)_{l\in\N}\subset \C^*$.\medskip
	
	From the above assertion, we see that $$\int f\Kmc_\theta^{i-1}\wedge\pi_1^*S\wedge \pi_2^*R=\int u_\theta^\Kmc dd^cf\wedge\big\langle\left(dd^cu\right)^{i-2}\wedge\pi_1^*S\wedge \pi_2^*R\big\rangle_C$$ converges as $\theta\to 0$, where $f$ is a smooth test $(n-i, n-i)$-form on $\Uf_j^{i-1}$.
\end{proof}

The following lemma was actually proved in \cite{Ahn25}. For reader's convenience, we write an adapted version of the proof.
\begin{lemma}\label{lem:conti_conti}
	Let $i\ge 2$ and $j\in J$. Let $S$ and $R$ and $(R_k)_{k\in\N}$ be as in Proposition \ref{prop:conti_usc_local}. Suppose that $\displaystyle \lim_{\delta\to 0}\nu_{S, j}^{i-1}(\delta)= 0$. Then, on $\Uf_j^{i-1}$, we have the continuity of the product as follows:
	{\begin{align*}
			\left\langle \left(dd^cu\right)^{i-1}\wedge\pi_1^*S\wedge \pi_2^*R_k\right\rangle_C\to \left\langle \left(dd^cu\right)^{i-1}\wedge\pi_1^*S\wedge \pi_2^*R\right\rangle_C
	\end{align*}}
	in the sense of currents.
\end{lemma}

\begin{proof}
	Let $\alpha$ be a smooth positive closed $(n-i+1, n-i+1)$-form on $\Uf_j^0$. Let $(\theta_l)_{l\in\N}\subset \C^*$ be a sequence converging to $0$. For each $l\in\N$, let
	\begin{align*}
		\Theta_l(k):=\int \chi_j^{i-1}u_{\theta_l}^\Kmc \alpha\wedge\left\langle\left(dd^cu\right)^{i-2}\wedge\pi_1^*S\wedge \pi_2^*R_k\right\rangle_C\quad\textrm{ for }k\in\N.
	\end{align*}
	Then $\left(\left(\Theta_l(k)\right)_{k\in\N}\right)_{l\in\N}$ is a collection of sequences. Notice that since currents in $K_j$ have uniformly bounded mass, the proof in Lemma \ref{lem:def_classic_prod} implies that $\left(\left(\Theta_k(l)\right)_{k\in\N}\right)_{l\in\N}$ is uniformly Cauchy with respect to $k$. That is, $\left(\left(\Theta_l(k)\right)_{k\in\N}\right)_{l\in\N}$ converges uniformly with respect to $k$. Hence, we can switch the order of limits $\displaystyle \lim_{k\to\infty}\lim_{l\to\infty}\Theta_l(k)=\lim_{l\to\infty}\lim_{k\to\infty}\Theta_l(k)$. Then, for a smooth test $(n-i, n-i)$-form $f$ on $\Uf_j^{i-1}$, we have
	\begin{align*}
		&\lim_{k\to\infty}\int_{\Uf_j^{i-1}}f\left\langle \left(dd^cu\right)^{i-1}\wedge\pi_1^*S\wedge \pi_2^*R_k\right\rangle_C=\lim_{k\to\infty}\lim_{l\to \infty}\int_{\Uf_j^{i-1}}f \Kmc_{\theta_l}^{i-1}\wedge\pi_1^*S\wedge \pi_2^*R_k\\
		&=\lim_{l\to \infty}\lim_{k\to\infty}\int_{\Uf_j^{i-1}}f \Kmc_{\theta_l}^{i-1}\wedge\pi_1^*S\wedge \pi_2^*R_k=\int_{\Uf_j^{i-1}}f\left\langle \left(dd^cu\right)^{i-1}\wedge\pi_1^*S\wedge \pi_2^*R\right\rangle_C
	\end{align*}
	
\end{proof}

\begin{proof}
	[Proof of Proposition \ref{prop:conti_usc_local}] From Lemma \ref{lem:conti_conti}, the function $\Fc_{S, j, \theta}^i$ is continuous on $\Cmc_{S, j}^{i-1}$. As $\Fc_{S, j}^i$ is a decreasing limit of $\Fc_{S, j, \theta}^i$, $\Fc_{S, j}^i$ is upper semi-continuous on $\Cmc_{S, j}^{i-1}$.
\end{proof}
%
\begin{proposition}\label{prop:char_conti}
	Let $j\in J$. The functions $\Fc_{S, j}^i$ are defined and continuous on $K_j$ for $i=1, \ldots, n$ with respect to the $*$-topology on $U_j^0$ if and only if $\displaystyle \lim_{\delta\to 0}\nu_{S, j}^i(\delta)=0$ for $i=1, \ldots, n$.
\end{proposition}
By the nature of the inductive definition, the proposition does not hold individually for each $i=1,\ldots ,n$ but collectively for $i=1,\ldots ,n$ together.
\begin{proof}
	Suppose that $\displaystyle \lim_{\delta\to 0}\nu_{S, j}^i(\delta)=0$ for $i=1, \ldots, n$. We first prove $\Cmc_{S, j}^i = K_j$ for every $i=1, \ldots, n$. We use induction on $i$. When $i=1$, from the compactness of $K_j$ and the continuity of $\Fc_{S, j}^1$, the finiteness is obvious. We assume that $\Cmc^{i-1}_{S,j}=K_j$. Let $R\in \Cmc^{i-1}_{S,j}$. Let $\delta_0\ll 1$.  Let $M_{K_j}>0$ be such that $\sup_{R'\in K_j}\|R'\|_j\le M_{K_j}$. Since smooth currents are dense in $K_j$, we can find a sequence $\left(R_k\right)_{k\in\N}$ of smooth currents in $\widetilde{K}_j$ such that $\displaystyle \lim_{k\to\infty}R_k=R$. By definition, we have $\displaystyle\Fc_{S, j}^i(R)=\lim_{\theta\to 0}\Fc_{S, j, \theta}^i(R)$. Also, since $\displaystyle \lim_{\delta\to 0}\nu_{S, j}^{i-1}(\delta)=0$, from Proposition \ref{prop:conti_usc_local}, we have
	\begin{align*}
		\Fc_{S, j}^i(R)&\ge \limsup_{k\to\infty}\Fc^i_{S, j}(R_k)=\limsup_{k\to\infty}\int \chi_j^iu(dd^cu)^{i-1}\pi_1^*S\wedge \pi_2^*R_k\wedge\omega^{n-i}\\
		&\ge \int_{\C^n\times\C^n\setminus \Delta_{\delta_0}} \chi_j^i u(dd^cu)^{i-1}\wedge\pi_1^*S\wedge\pi_2^*R\wedge\omega^{n-i} - M_{K_j}\nu_{S, j}^i(\delta_0).
	\end{align*}
	The first integral is finite because $u$ is smooth on $\C^n\times\C^n\setminus \Delta_{\delta_0}$ and $\chi_j^i$ is compactly supported. Hence, we have $\Cmc_S^i = K_j$.\medskip 
	
	We prove continuity. Let $i\in\{1, \ldots, n\}$ and $0<\delta\ll 1$. Since the space $K_j$ is metrizable, it is enough to show that for a current $R\in K_j$ with $\|R\|\le 1$ and a sequence $\left(R_k\right)_{k\in\N}\subset K_j$ such that $\displaystyle \lim_{k\to\infty}R_k=R$ and $\|R_k\|\le 1$, we have
	\begin{align*}
		\lim_{k\to\infty}\Fc_{S, j}^i(R_k)=\Fc_{S, j}^i(R).
	\end{align*}
	From the upper-semicontinuity of $\Fc_{S, j}^i$ on $K_j$ as in Proposition \ref{prop:conti_usc_local}, it is enough to show that $\displaystyle \lim_{k\to\infty}\left(\Fc_{S, j}^i(R)-\Fc_{S, j}^i(R_k)\right)\le 0$. From the negativity of the integrand, we have
	\begin{align*}
		\Fc_{S, j}^i(R)\le \int \chi_{\Delta, \delta}\chi_j^i u(dd^cu)^{i-1}\wedge\pi_1^*S\wedge\pi_2^*R\wedge\omega^{n-i}
	\end{align*}
	and as previously, we have
	\begin{align*}
		\Fc_{S, j}^i(R)&\ge \int \chi_{\Delta, \delta}\chi_j^i u(dd^cu)^{i-1}\wedge\pi_1^*S\wedge\pi_2^*R_k\wedge\omega^{n-i} - \nu_{S, j}^i(\delta),
	\end{align*}
	where $\chi_{\Delta, \delta}:\C^n\times\C^n\to [0, 1]$ is a smooth function such that $\chi_{\Delta, \delta}\equiv 1$ for $|x-y|\ge \delta$ and $\chi_{\Delta, \delta}=0$ near $\Delta$. Hence, we have
	\begin{align*}
		&\Fc_{S, j}^i(R)-\Fc_{S, j}^i(R_k)\le \int\chi_{\Delta, \delta} \chi_j^i u(dd^cu)^{i-1}\wedge\pi_1^*S\wedge\pi_2^*\left(R-R_k\right)\wedge\omega^{n-i} + \nu_{S, j}^i(\delta).
	\end{align*}
	Since $u$ is smooth outside $\Delta$ and $\chi_j^i$ is compactly supported, the first integral converges to $0$ as $k\to\infty$. Hence, we see that for every $\delta>0$, we have
	\begin{align*}
		\lim_{k\to\infty}\left(\Fc_{S, j}^i(R)-\Fc_{S, j}^i(R_k)\right)\le \nu_{S,j}^i(\delta).
	\end{align*}
	The assumption $\displaystyle \lim_{\delta\to 0}\nu_{S, j}^i(\delta)= 0$ implies $\displaystyle \lim_{k\to\infty}\big(\Fc_{S, j}^i(R)-\Fc_{S, j}^i(R_k)\big)\le 0$ as desired.\bigskip
	
	We prove the other direction. Let $\chi_\Delta:\R\to [0, 1]$ be a smooth function such that $\chi_\Delta(t)\equiv 0$ if $|t|\le 1$ and $\chi_\Delta(t)\equiv 1$ if $|t|\ge 2$. For $k\in\N$, we define
	\begin{align*}
		\Gc_{S,j,k}^i(R):=\int\chi_\Delta\left(k|x-y|\right)\chi_j^iu \left(dd^cu\right)^{i-1}\wedge\pi_1^*S\wedge \pi_2^*R\wedge \omega^{n-i}.
	\end{align*}
	Then, it is continuous on $K_j$. Since $\Fc_{S, j}^i$ is defined and continuous on the entire $K_j$, its integrand has no mass on $\Delta$. Hence, $\Gc_{S, j, k}^i(R)$ decreasingly converges to $\Fc_{S, j}^i(R)$ as $k\to\infty$. Since for $R\in \widetilde{K}_j$, we have
	\begin{align*}
		0&\ge \int_{\Delta_\delta}\chi_j^iu \left(dd^cu\right)^{i-1}\wedge\pi_1^*S\wedge \pi_2^*R\wedge \omega^{n-i}\ge\Fc_{S, j}^i(R)-\Gc_{S, j, k}^i(R),
	\end{align*}
	where $k<1/\delta$. For the proof, we need to prove that $\Gc_{S, j, k}^i$ converges to $\Fc_{S, j}^i$ uniformly on $K_j$. This is from Dini's theorem since $\Fc_{S, j}^i$ is assumed to be continuous on $K_j$ and $K_j$ is compact.
\end{proof}

The proof of the second assertion can be actually written as a version of uniform convergence as follows:
\begin{proposition}\label{prop:unif_conv_poten_sup}
	Let $j\in J$. Suppose that $\displaystyle \lim_{\delta\to 0}\nu_{S, j}^i(\delta)=0$ for $i=1, \ldots, n$. (Or equivalently, suppose that $\Fc_{S, j}^i$ is defined and continuous on $K_j$ for $i=1, \ldots, n$.) Then, the convergence of $\Fc_{S, j, \theta}^i$ to $\Fc_{S, j}^i$ is uniform on $K_j$ for $i=1, \ldots, n$.
\end{proposition}
As a direct corollary, we obtain
\begin{proposition}\label{prop:uniform_classical_prod}
	Let $\left(K_j\right)_{j\in J}$ be as in Definition \ref{def:conti_loc_supftn}. Let $S\in\Cc_p(X)$ admit continuous local potential functionals on $\left(K_j\right)_{j\in J}$. Then, for $j\in J$, the convergence of $\pi_1^*S\wedge \pi_2^*R\wedge \Kmc_\theta^i$ and $(dd^c u_\theta^\Tmc)\wedge\big\langle\pi_1^*S\wedge \pi_2^*R\wedge (dd^cu)^{i-1}\big\rangle_C$ to $\big\langle \pi_1^*S\wedge \pi_2^*R\wedge (dd^cu)^i\big\rangle_C$ as $|\theta|\to 0$ is uniform on $\Uf_j^i$ with respect to $R\in K_j$. 
\end{proposition}

Using Propositions \ref{prop:char_conti} and \ref{prop:unif_conv_poten_sup}, we can characterize continuous local potential functionals as in \cite{DNV}.
\begin{theorem}\label{thm:main_char}
	Let $X$ be a complex manifold of dimension $n$ with localizing data $\big(\big(U_j^i, \xi_j^i\big)\big)_{i=0, 1, \ldots, n, j\in J}$. Let $\left(K_j\right)_{j\in J}$ be as in Definition \ref{def:conti_loc_supftn}. Let $S\in\Cc_p(X)$ and $\left(\nu_{S,j}^i\right)_{i=1, \ldots, n, j\in J}$ the associated functions as above. Then, $S$ admits continuous local potential functionals on $\left(K_j\right)_{j\in J}$ if and only if $\displaystyle \lim_{\varepsilon\to 0} \nu_{S,j}^i(\varepsilon)=0$ for $i=1, \ldots, n$ and $j\in J$. In this case, for each $i=1, \ldots, n$ and $j\in J$, $\Fc_{S, j, \theta}^i$ uniformly converges to $\Fc_{S, j}^i$ on $K_j$ as $|\theta|\to 0$.
\end{theorem}

In \cite{Ahn25}, Condition (I) for the Dinh-Sibony product was introduced. In this section, we only consider the intersection of two positive closed currents.
\begin{definition}\label{def:conditionI}
	Let $D$ be a bounded simply connected domain with smooth boundary in $\C^n$. Let $S_i\in \Cc_{s_i}(X)$ for $i=1, 2$, where $1\le s:=s_1+s_2\le n$. We say that $S_1$ and $S_2$ satisfy Condition (I) if
	\begin{align*}
		u\in L_\loc^1\left(\left\langle \pi_1^*S_1\wedge \pi_2^*\left(S_2\wedge \omega_\euc^{2n-s-i+1}\right)\wedge (dd^cu)^{i-1} \right\rangle_C\right)
	\end{align*}
	in $D^2$ inductively from $i=1$ through $n$.
\end{definition} 


\begin{proposition}\label{prop:ConditionI_under_conti}
	Let $X$, $\left(\left(U_j^i, \xi_j^i\right)\right)_{i=0, 1, \ldots, n, j\in J}$ and $\left(K_j\right)_{j\in J}$ be as above. Let $S\in\Cc_p(X)$ and $\left(\nu_{S,j}^i\right)_{i=1, \ldots, n, j\in J}$ the associated functions as above. Let $j\in J$ and $q\in\{1, \ldots, n-p\}$. Let $R\in\Cc_q(X)$ be such that $R\wedge \alpha\in K_j$ for some smooth positive closed $(n-p-q+1, n-p-q+1)$-form $\alpha$ on $U_j^0$. Suppose that $\Fc_{S, j}^i$ is defined and continuous on $K_j$ for $i=1, \ldots, n$. Then, $S$ and $R$ satisfy Condition (I) on $U_j^n$.
\end{proposition}


%
%
Together with \cite[Theorem 1.1]{Ahn25}, we get

\begin{theorem}\label{thm:intersection}
	We assume $X$, $\left(\left(U_j^i, \xi_j^i\right)\right)_{i=0, 1, \ldots, n, j\in J}$, $\left(K_j\right)_{j\in J}$, $S\in\Cc_p(X)$ and $\big(\nu_{S,j}^i\big)_{i=1, \ldots, n, j\in J}$ as above. Let $q\in \{1, \ldots, n-p\}$ and $R\in\Cc_q(X)$ be such that for each $j\in J$, $R\wedge \alpha_j\in K_j$ for some smooth positive closed $(n-p-q+1, n-p-q+1)$-form $\alpha_j$ on $U_j^0$. Suppose that $S$ admits continuous local potential functionals on $\left(K_j\right)_{j\in J}$. Then, the Dinh-Sibony product of $S$ and $R$ is well defined on $X$.
\end{theorem}

We end this section by generalizing the so-called domination principle introduced in \cite{DNV}. Together with Section [...] below, when restricted to compact K\"ahler manifolds, it coincides with the domination principle for continuous superpotentials as in \cite{DNV}. The proof is straightforward as in \cite{DNV}.
\begin{theorem}\label{thm:domination}
	Let $X$ be a complex manifold of $n$. Let $S$ and $S'$ be positive closed $(p, p)$-currents on $S$ such that $S'\le S$ in the sense of currents. Suppose that $S$ admits continuous local potential functionals. Then, $S'$ also admits continuous local potential functionals.
\end{theorem}

\begin{example}
	Classical pluripotential theory proves that $S=\bigwedge_{i=1}^pdd^c u_i$ on a domain in $\C^n$ admits continuous local potential functionals, where $u_i$'s are bounded plurisubharmonic functions.
\end{example}

Another example is a positive closed current with continuous superpotentials.

\begin{example}
	Let $X$ and $Y$ be two complex manifolds of dimension $n$. Let $f:X\to Y$ be a surjective finite-to-one holomorphic map. Then, for any smooth positive closed $(p, p)$-form $\varphi$, $f_*\varphi$ admits continuous local potential functionals.
\end{example}

\section{Associativity of the Dinh-Sibony product}\label{sec:associativity}
In this section, we look into the associativity property for the Dinh-Sibony product in connection with the continuity of local potential functionals.\medskip

Let $X$ be a complex manifold of dimension $n$. Let $\big(\big(U_j^i, \xi_j^i\big)\big)_{i=0, 1, \ldots, n, j\in J}$ be its localizing data. For each $j\in J$, we consider 
$$K_j=\{R\in\Cc_{n-p+1}(U_1^1; U_j^0): \|R\|_{*, j}\le 1\}.$$
Let $s_1\in\{1, \ldots, n\}$ and $S_1\in\Cc_{s_1}(X)$ be such that $S_2$ admits continuous local potential functionals on $\left(K_j\right)_{j\in J}$. Let $\left(\nu_{S_1,j}^i\right)_{i=1, \ldots, n, j\in J}$ be the associated functions as in Section \ref{sec:conti}. Let $s_2, s_3\in \{1, \ldots, n\}$ be such that $1\le s_1+s_2+s_3\le n$ and let $S_2\in \Cc_{s_2}(X)$ and $S_3\in\Cc_{s_3}(X)$. Theorem \ref{thm:intersection} implies that $\big(S_1\wedge S_2\big)_{DS}$ is well defined on $X$.\medskip 

Below is Theorem \ref{thm:associativity}, which is the main result of this work.
\begin{theorem}
	Suppose that $S_1$ admits continuous local potential functionals on $\left(K_j\right)_{j\in J}$ and that $S_2$ and $S_3$ satisfy Condition (I) on $U_j^0$ for each $j\in J$. Then, the Dinh-Sibony product $\left((S_1\wedge S_2)_{DS}\wedge S_3\right)_{DS}$ of $\big(S_1\wedge S_2\big)_{DS}$ and $S_3$, and the Dinh-Sibony product $(S_1\wedge S_2\wedge S_3)_{DS}$ of $S_1$, $S_2$ and $S_3$ are well defined for each $j\in J$, and we have
	\begin{align*}
		\left(S_1\wedge (S_2\wedge S_3)_{DS}\right)_{DS}=\left((S_1\wedge S_2)_{DS}\wedge S_3\right)_{DS}=(S_1\wedge S_2\wedge S_3)_{DS}\textrm{ on }X.
	\end{align*}
\end{theorem}

For the proof, we introduce some notations. We write $\Xf:=X_1\times X_2\times X_3$, $\Xf_{12}:=X_1\times X_2$ and $\Xf_{23}:=X_2\times X_3$, where $X_1$, $X_2$, $X_3$ are copies of $X$. We consider the following canonical projection maps.

\begin{center}
	\begin{tikzpicture}
		\node 		(X123) 				at (0, 0) {$\Xf:=X_1\times X_2\times X_3$};
		\node      	(X23)               at (1.5, -1.5) 			{$\Xf_{23}:=X_2\times X_3$};
		\node       (X12)       		at (-1.5, -1.5) 			{$\Xf_{12}:=X_1\times X_2$};
		\node       (X1)      			at (-3, -3) 			{$X_1$};
		\node 		(X2)				at (0, -3)			{$X_2$};
		\node		(X3)				at (3, -3)			{$X_3$};
		\draw[->] (X123) -- (X12) node[midway, left] {$\pi_{12}$};
		\draw[->] (X123) -- (X23) node[midway, right] {$\pi_{23}$};
		\draw[->] (X12) -- (X1) node[midway, left] {$\pi^{12}_1$};
		\draw[->] (X23) -- (X2) node[midway] {$\pi^{23}_2\,\,\,\,\,\,\,\,$};
		\draw[->] (X12) -- (X2) node[midway, left] {$\pi^{12}_2$};
		\draw[->] (X23) -- (X3) node[midway, right] {$\pi^{23}_3$};
	\end{tikzpicture}
\end{center}
As we will be working in a local situation, let $(x_1, x_2, x_3)$ denote local coordinates in $\C^{3n}=\C^n\times\C^n\times\C^n$. We let $\Delta_{12}:=\{(x_1, x_2, x_3)\in \C^{3n}: x_1=x_2\}$, $\Delta_{23}:=\{(x_1, x_2, x_3)\in \C^{3n}: x_2=x_3\}$ and $\Delta:=\{(x, x, x)\in \C^{3n}\}$. We denote by $u:=\log|x_1-x_2|$ and $v:=|x_2-x_3|$. As in Subsection \ref{subsec:local_pot_supftns}, We correspondingly define $u_\theta^\Kmc$, $u_\theta^\Tmc$, $v_\theta^\Kmc$, $v_\theta^\Tmc$ and $\Kmc_{u, \theta}^i$, $\Kmc_{v, \theta}^i$ for $i=1, \ldots, n$, respectively. We write the Euclidean K\"ahler forms $\omega_{12}:=dd^c\left(|x_1|^2+|x_2|^2\right)$, $\omega_{23}:=dd^c\left(|x_2|^2+|x_3|^2\right)$ and $\omega:=dd^c\left(|x_1|^2+|x_2|^2+|x_3|^2\right)$ on their respective spaces.\medskip 

For the estimates near $\Delta_{12}$ and $\Delta_{23}$, we consider the following functions. The lemmas are direct from computations.
\begin{lemma}\label{lem:cut_off_12}
	For $k\in\N$ with $k\gg 1$, the function $\chi_k^u:\Xf_{12}\to [0, 1]$ defined by $$\chi^u_k:=\frac{u_{1/k^2}^\Kmc-u_{1/k}^\Kmc}{\log k}$$ is smooth with support in $|x_1-x_2|\le e/k$ such that $\chi_k \equiv 1$ on $\{|x_1-x_2|\le \frac{1}{ek^2}\}$.
\end{lemma}

\begin{lemma}\label{lem:cut_off_23}
	For $k\in\N$ with $k\gg 1$, the function $\chi_k^v:\Xf_{23}\to [0, 1]$ defined by $$\chi^v_k:=\frac{v_{1/k^2}^\Kmc-v_{1/k}^\Kmc}{\log k}$$ is smooth with support in $|x_2-x_3|\le e/k$ such that $\chi_k \equiv 1$ on $\{|x_2-x_3|\le \frac{1}{ek^2}\}$.
\end{lemma}

We prove Theorem \ref{thm:associativity} in two propositions: Proposition \ref{prop:1st_part} and Proposition \ref{prop:2nd_part}. Due to \cite[Proposition 3.5 and Remark 4.9]{DS18}, we may assume that $s_1+s_2+s_3=n$.\medskip
We first prove Proposition \ref{prop:1st_part} in several steps.
\begin{proposition}\label{prop:1st_part}
	Under the assumptions of Theorem \ref{thm:associativity}, for each $j\in J$, the Dinh-Sibony product $\left((S_1\wedge S_2)_{DS}\wedge S_3\right)_{DS}$ of $\big(S_1\wedge S_2\big)_{DS}$ and $S_3$ is well defined on $U_j^n$ and we have
	\begin{align*}
		\left(S_1\wedge (S_2\wedge S_3)_{DS}\right)_{DS}=\left((S_1\wedge S_2)_{DS}\wedge S_3\right)_{DS}\textrm{ on }U_j^n.
	\end{align*}
\end{proposition}

\begin{proof}
	For the proof of Proposition \ref{prop:1st_part}, we need the following two claims, which will be proved separately for clarity. We consider a fixed $j\in J$.\medskip 
	
	Let $\varphi$ be a positive smooth test function form on $\Uf_j^n\subset \Xf_{23}$ since every smooth test function can be written as a difference of two smooth positive test functions. \smallskip
	
	\noindent {\bf Claim 1.} For the function
	{\begin{align*}
			I'_i(k,\theta, \theta'):=\int_{\Xf_{23}}\varphi\chi_k^v(dd^cv_\theta^\Tmc)\wedge \Kmc_{v, \theta'}^{i-1}\wedge \Big(\pi_2^{23}\Big)^*(S_1\wedge S_2)_{DS}\wedge\Big(\pi_3^{23}\Big)^*S_3\wedge\omega_{23}^{n-i},
		\end{align*}
	} the limit $\displaystyle \lim_{\theta\to 0}\lim_{\theta'\to 0}I'_i$ exists for $1\ll k$ and $i=1, \ldots, n$, and we have $$\lim_{k\to\infty}\lim_{\theta\to 0}\lim_{\theta'\to 0}I'_i=0\textrm{ for }i=1, \ldots, n-1.$$\smallskip
	
	\noindent {\bf Claim 2.} The shadow of the tangent currents of $\left(\pi_2^{23}\right)^*\big(S_1\wedge S_2\big)_{DS}\wedge\left(\pi_3^{23}\right)^*S_3$ equals $\big(S_1\wedge \big(S_2\wedge S_3\big)_{DS}\big)_{DS}$ in $U_j^n$.\medskip
	
	Once the two claims are proved, the proof works in this way. For $k\in\N$, outside the support of $\chi_k^v$, $(dd^cv_\theta^\Tmc)\wedge \Kmc_{v, \theta'}^{i-1}$ converges uniformly in $\theta'$ and $\theta$. Hence, together with \cite[Lemma 2.16 and Theorem 2.15]{Ahn25}, the existence of the limit in Claim 1 proves tangent currents of $\left(\pi_2^{23}\right)^*\big(S_1\wedge S_2\big)_{DS}\wedge\left(\pi_3^{23}\right)^*S_3$ along $\Delta_{23}$ exist.\medskip
	
	Next, the limit $\displaystyle\lim_{k\to\infty}\lim_{\theta\to 0}\lim_{\theta'\to 0}I'_i=0$ for $i=1, \ldots, n-1$ in Claim 1 implies that the $h$-dimension of the tangent currents is minimal. Together with \cite[Proposition 3.10]{Ahn25}, Claim 2 means that there exists a unique tangent current and the Dinh-Sibony product $\big(S_1\wedge S_2\big)_{DS}$ and $S_3$ is well defined on $U_j^n$ and we have the desired equality.
\end{proof}


\subsection{Proof of Claim 1} The proof of the existence of the limit $\displaystyle \lim_{k\to\infty}\lim_{\theta'\to 0}I'_i$ for $i=1, \ldots, n-1$ is proved in the course of the proof of $\displaystyle \lim_{k\to\infty}\lim_{\theta\to 0}\lim_{\theta'\to 0}I'_i=0$ for $i=1, \ldots, n-1$. The case of $i=n$ can be dealt with in the same way. So, it suffices to only consider $\displaystyle \lim_{k\to\infty}\lim_{\theta\to 0}\lim_{\theta'\to 0}I'_i=0$ for $i=1, \ldots, n-1$.\medskip

Let $\eta, \eta'\in\C^*$ be such that $|\eta|, |\eta'|\ll 1$. For $S_1\in\Cc_{s_1}(X)$, $S_3\in\Cc_{s_3}(X)$, $i=1, \ldots, n$, and $m=i, \ldots, n, n+1$, we denote by
\begin{align*}
	&\Uc^{m, i}_{1, \eta, \eta'}:=u_\eta^\Kmc\wedge\Kmc_{u, \eta'}^{i-1}\wedge \Big(\pi_1^{12}\Big)^*S_1\wedge \omega_{12}^{m-i},\\
	&\Sc^{m, i}_{3, \theta, \theta'}:=(dd^cv_\theta^\Tmc)\wedge \Kmc_{v, \theta'}^{i-1}\wedge\Big(\pi_3^{23}\Big)^*S_3\wedge\omega_{23}^{m-i},\\
	&\left(\Tc_\theta\Kc_{\theta'}\right)^{m, i}(S_3):=\left(\pi_2^{23}\right)_*\Big(\left(\pi_3^{23}\right)^*S_3\wedge \left(dd^c v_\theta^\Tmc\right)\wedge \Kmc_{v, \theta'}^{i-1}\wedge\omega_{23}^{m-i} \Big),\\
	&\left(\Tc_\theta\Kc_{\theta'}\right)_{\theta_0}^{m, i}(S_3):=\left(\pi_2^{23}\right)_*\Big(\left(\pi_3^{23}\right)^*S_3\wedge \left(dd^c \chi\left(v_\theta^\Tmc-\log|\theta_0|\right)\right)\wedge \Kmc_{v, \theta'}^{i-1}\wedge\omega_{23}^{m-i} \Big).
\end{align*}
Notice the relationship $\left(\Tc_\theta\Kc_{\theta'}\right)^{m, i}(S_3):=\left(\pi_2^{23}\right)_*\Sc^{m, i}_{3, \theta, \theta'}$.\medskip

The transformed current $\left(\Tc_\theta\Kc_{\theta'}\right)^{m, i}(S_3)$ is positive but may not be closed. For closedness, we add a positive smooth form to $\left(\Tc_\theta\Kc_{\theta'}\right)^{m, i}(S_3)$.

\begin{lemma}\label{lem:semi-regular_transform}
	Let $\theta_0\in\C^*$ be such that $|\theta_0|\ll 1$. For $\theta, \theta'\in \C^*$ with $|\theta| <\frac{|\theta_0|}{\sqrt{2}e}$ and $|\theta'|<\frac{|\theta_0|}{\sqrt{2}e^2}$, $\left(\Tc_\theta\Kc_{\theta'}\right)^{m,i}(S_3)$ and $\left(\Tc_\theta\Kc_{\theta'}\right)_{\theta_0}^{m,i}(S_3)$ are smooth positive $(s_3+m-n, s_3+m-n)$-forms and we have
	\begin{align*}
		\left\|\left(\Tc_\theta\Kc_{\theta'}\right)_{\theta_0}^{m,i}(S_3)\right\|_\infty \le c_{TK}\|S_3\||\theta_0|^{-2i},
	\end{align*}
	where $c_{TK}$ is a positive constant independent of $S_3$, $\theta_0$, $\theta$ and $\theta'$.
\end{lemma}

\begin{proof}
	Since the support of $dd^c \chi\left(v_\theta^\Tmc-\log|\theta_0|\right)$ does not meet the neighborhood $\left\{|x_2-x_3|<\frac{|\theta_0|}{\sqrt{2e}}\right\}$ of $\Delta_{23}$, the form $\left(dd^c \chi\left(v_\theta^\Tmc-\log|\theta_0|^2\right)\right)\wedge \left(dd^c v_{\theta'}^\Kmc\right)\wedge(dd^cv)^{i-2}$ is smooth and the the desired estimate is obtained from direct computations.
\end{proof}

As a corollary, we have
\begin{lemma}\label{lem:DSHness}
	Let $\theta_0\in\C^*$ be such that $|\theta_0|\ll 1$. There exists a constant $c_{TK}'>0$ independent of $S_3$, $\theta_0$, $\theta$ and $\theta'$ such that for $\theta, \theta'\in \C^*$ with $|\theta| <\frac{|\theta_0|}{\sqrt{2}e}$ and $|\theta'|<\frac{|\theta_0|}{\sqrt{2}e^2}$, the current
	\begin{align*}
		\left(\Tc_\theta\Kc_{\theta'}\right)^{m,i}(S_3)-\left(\Tc_\theta\Kc_{\theta'}\right)_{\theta_0}^{m,i}(S_3)+c_{TK}'\|S_3\||\theta_0|^{-2i}\omega_\euc^{s_3+m-n}
	\end{align*}
	is a smooth positive $(s_3+m-n, s_3+m-n)$-current closed in $U_j^1$. Its mass can be bounded independently of $\theta_0$, $\theta$ and $\theta'$. The form $\left(\Tc_\theta\Kc_{\theta'}\right)_{\theta_0}^{m,i}(S_3)-c_{TK}'\|S_3\||\theta_0|^{-2i}\omega_\euc^{s_3+m-n}$ is smooth and we have
	\begin{align*}
		\|\left(\Tc_\theta\Kc_{\theta'}\right)_{\theta_0}^{m,i}(S_3)+c_{TK}'\|S_3\||\theta_0|^{-2i}\omega_\euc^{s_3+m-n}\|_\infty \le c''_{TK} \|S_3\||\theta_0|^{-2i},
	\end{align*}
	where $c''_{TK}$ is a positive constant independent of $S_3$, $\theta$, $\theta_0$ and $\theta'$.
\end{lemma}

\begin{proof}
	We choose $c_{TK}'$ to be a constant multiple of $c_{TK}$ as in Lemma \ref{lem:semi-regular_transform}. The positivity is clear from Lemma \ref{lem:semi-regular_transform}. It suffices to check the closedness of $\left(\Tc_\theta\Kc_{\theta'}\right)^{m,i}(S_3)-\left(\Tc_\theta\Kc_{\theta'}\right)_{\theta_0}^{m,i}(S_3)$ in $U_j^1$. It just comes from the fact that the support of $dd^c \left(v_\theta^\Kmc-\chi\left(v_\theta^\Tmc-\log|\theta_0|\right)\right)$ uniformly shrinks to $\Delta_{23}$ as $\theta_0\to 0$. So, for $\theta_0\in\C^*$ with sufficiently small $|\theta_0|$, the closedness of $\left(\Tc_\theta\Kc_{\theta'}\right)^{m,i}(S_3)-\left(\Tc_\theta\Kc_{\theta'}\right)_{\theta_0}^{m,i}(S_3)$ in $U_j^1$ is obtained. Lemma \ref{lem:semi-regular_transform} gives the last estimate.
\end{proof}

%

\begin{lemma}\label{lem:existence}
	Let $i\in\{1, \ldots, n\}$. The following limit exists:
	{\begin{align*}
			&\lim_{\theta\to 0}\lim_{\theta'\to 0}\lim_{\eta\to 0}\lim_{\eta'\to 0}\int_{\Xf_{12}}\chi_j^n\Uc_{1, \eta, \eta'}^{n, n}\wedge \Big(\pi_2^{12}\Big)^*\left(S_2\wedge (\Tc_\theta\Kc_{\theta'})^{n+1, i}(S_3)\right)
	\end{align*}}
\end{lemma}

\begin{proof}
	Lemma \ref{lem:DSHness} says that the current $(\Tc_\theta\Kc_{\theta'})^{n+1, i}(S_3)$ can be written as the sum of a positive closed current and a smooth form in $U_j^1$ as follows:
	\begin{align*}
		[(\Tc_\theta\Kc_{\theta'})^{n+1, i}(S_3)-A(\theta_0, \theta, \theta')] + A(\theta_0, \theta, \theta'),
	\end{align*}
	where $A(\theta_0, \theta, \theta'):=\left(\Tc_\theta\Kc_{\theta'}\right)_{\theta_0}^{n+1,i}(S_3)+c'_{TK}\|S_3\||\theta_0|^{-2i}\omega_\euc^{s_3+1}$. The smooth form $A(\theta_0, \theta, \theta')$ is smooth in $\theta$ and $\theta'$, and $\|A\|_\infty$ can be bounded in terms of $\|S_3\|$ and $\theta_0$, uniformly in $\theta$ and $\theta'$ as in Lemma \ref{lem:DSHness}. For each $\theta$ and $\theta'$, we have
	{\begin{align}
			\notag&\int_{\Xf_{12}}\chi_j^n\Uc_{1, \eta, \eta'}^{n, n}\wedge \Big(\pi_2^{12}\Big)^*\left(S_2\wedge (\Tc_\theta\Kc_{\theta'})^{n+1, i}(S_3)\right)\\
			\label{eq:lem4.7-1}&=\int_{\Xf_{12}}\Big[\chi_j^n\Uc_{1, \eta, \eta'}^{n, n}\wedge \Big(\pi_2^{12}\Big)^*\Big(S_2\wedge\big[\left(\Tc_\theta\Kc_{\theta'}\right)^{n+1,i}(S_3)-A(\theta_0, \theta, \theta')\big]\Big)\\
			\label{eq:lem4.7-2}&\quad\quad\quad+\chi_j^n\Uc_{1, \eta, \eta'}^{n, n}\wedge \Big(\pi_2^{12}\Big)^*\Big(S_2\wedge\big[A(\theta_0, \theta, \theta')\big]\Big) \Big].
	\end{align}}
	
	For \eqref{eq:lem4.7-2}, since $A(\theta_0, \theta, \theta')$ is smooth, we have
	\begin{align*}
		&\lim_{\eta\to 0}\lim_{\eta'\to 0}\int_{\Xf_{12}}\chi_j^n\Uc_{1, \eta, \eta'}^{n, n}\wedge \Big(\pi_2^{12}\Big)^*\Big(S_2\wedge A(\theta_0, \theta, \theta')\Big)\\
		&=\lim_{\eta\to 0}\int_{\Xf_{12}}\chi_j^nu_\eta^\Kmc\Big\langle(dd^cu)^{n-1}\wedge \Big(\pi_1^{12}\Big)^*S_1\wedge \Big(\pi_2^{12}\Big)^*S_2\Big\rangle_C\wedge\Big(\pi_2^{12}\Big)^*A(\theta_0, \theta, \theta').
	\end{align*}
	Since the continuity of the local potential functionals of $S_1$ implies Condition (I) for $S_1$ and $S_2$, Lemma \ref{lem:induction} implies the convergence as $\eta'\to 0$. The smooth form $A(\theta_0, \theta, \theta')$ can be bounded by $\omega_\euc^{s_3+1}$ up to a multiplicative constant independent of $\theta$ and $\theta'$. Again, the continuity of the local potential functionals of $S_1$ implies the convergence as $\eta\to 0$. The regularity of $A(\theta_0, \theta, \theta')$ tells us that the limit of the last limit as $\theta'\to 0$ and then $\theta\to 0$ converges.\medskip
	
	For \eqref{eq:lem4.7-1}, since $S_2\wedge\big[\left(\Tc_\theta\Kc_{\theta'}\right)^{n+1,i}(S_3)-A(\theta_0, \theta, \theta')\big]$ is positive closed and its mass is bounded independently of $\theta$ and $\theta'$ from the same argument as in Lemma \ref{lem:DSHness}, the continuity of the local potential functionals of $S_1$ implies the following convergence:
	\begin{align*}
		&\lim_{\eta\to 0}\lim_{\eta'\to 0}\int_{\Xf_{12}}\chi_j^n\Uc_{1, \eta, \eta'}^{n, n}\wedge \Big(\pi_2^{12}\Big)^*\Big(S_2\wedge\big[\left(\Tc_\theta\Kc_{\theta'}\right)^{n+1,i}(S_3)-A(\theta_0, \theta, \theta')\big]\Big)\\
		&\quad\quad=\lim_{\eta\to 0} \Fc_{S_1, j, \eta}^n\Big(S_2\wedge\big[\left(\Tc_\theta\Kc_{\theta'}\right)^i(S_3)-A(\theta_0, \theta, \theta')\big]\Big)=\Fc_{S_1, j}^n\Big(S_2\wedge\big[\left(\Tc_\theta\Kc_{\theta'}\right)^i(S_3)-A(\theta_0, \theta, \theta')\big]\Big).
	\end{align*}
	The continuity of the local potential functionals of $S_1$ again implies the convergence of the last value as $\theta'\to 0$ and then $\theta\to 0$ as desired.
\end{proof}

We consider the existence of the following limit:
\begin{align}
	\label{eq:nearDelta23}&\lim_{\theta\to 0}\lim_{\theta'\to 0}\lim_{\eta\to 0}\lim_{\eta'\to 0}\int_{\Xf_{12}}\chi_j^n\Uc_{1, \eta, \eta'}^{n, n}\wedge \Big(\pi_2^{12}\Big)^*\Big[S_2\wedge\Big(\pi_2^{23}\Big)_*\Big[\left(\chi_k^v\right)^m\Sc^{n+1, i}_{3, \theta, \theta'}\Big]\Big].
\end{align}
Observe that $\Big(\pi_2^{23}\Big)_*\Big[\big(1-\left(\chi_k^v\right)^m\big)(dd^cv_\theta^\Tmc)\wedge \Kmc_{v, \theta'}^{i-1}\wedge\Big(\pi_3^{23}\Big)^*S_3\wedge\omega_{23}^{n-i+1}\Big]$ is smooth and converges uniformly as $\theta'\to 0$ and $\theta\to 0$ since the support of $1-\left(\chi_k^v\right)^m$ does not intersect $\Delta_{23}$. Hence, the continuity of local potential functionals of $S_1$ implies that the following limit exists:
{\begin{align}
		\label{eq:test}&\lim_{\theta\to 0}\lim_{\theta'\to 0}\lim_{\eta\to 0}\lim_{\eta'\to 0}\int_{\Xf_{12}}\chi_j^n\Uc_{1, \eta, \eta'}^{n, n}\wedge \Big(\pi_2^{12}\Big)^*\Big[S_2\wedge\Big(\pi_2^{23}\Big)_*\Big[\big(1-\left(\chi_k^v\right)^m\big)\Sc_{3, \theta, \theta'}^{n+1, i}\Big]\Big].
\end{align}}
The desired limit \eqref{eq:nearDelta23} is actually the difference of the limits in Lemma \ref{lem:existence} and \eqref{eq:test} and so exists.

\begin{lemma}\label{lem:vanish}
	Let $i\in\{1, \ldots, n-1\}$ and $m\in\N$. Then, we have
	\begin{align*}
		&\lim_{k\to\infty}\lim_{\theta\to 0}\lim_{\theta'\to 0}\lim_{\eta\to 0}\lim_{\eta'\to 0}\int_{\Xf_{12}}\chi_j^n\Uc_{1, \eta, \eta'}^{n, n}\wedge \Big(\pi_2^{12}\Big)^*\Big[S_2\wedge\Big(\pi_2^{23}\Big)_*\Big[\left(\chi_k^v\right)^m\Sc_{3, \theta, \theta'}^{n+1, i}\Big]\Big]=0
	\end{align*}
\end{lemma}

\begin{proof}
	Let $\varepsilon>0$ be given. We want to find $N_\varepsilon\in\N$ such that whenever $k\ge N_\varepsilon$, we have
	\begin{align*}
		&\bigg|\lim_{\theta\to 0}\lim_{\theta'\to 0}\lim_{\eta\to 0}\lim_{\eta'\to 0}\int_{\Xf_{12}}\chi_j^n\Uc_{1, \eta, \eta'}^{n, n}\wedge \Big(\pi_2^{12}\Big)^*\Big[S_2\wedge\Big(\pi_2^{23}\Big)_*\Big[\left(\chi_k^v\right)^m\Sc_{3, \theta, \theta'}^{n+1, i}\Big]\Big]\bigg|<\varepsilon
	\end{align*}

	According to Lemma \ref{lem:DSHness}, we can write $S_2\wedge \left(\Tc_\theta\Kc_{\theta'}\right)^{n+1, i}(S_3)$ into
	\begin{align*}
		&\left[S_2\wedge \left(\Tc_\theta\Kc_{\theta'}\right)^{n+1, i}(S_3)-S_2\wedge A(\theta_0, \theta, \theta')\right]+S_2\wedge A(\theta_0, \theta, \theta').
	\end{align*}
	Choosing $0<|\theta_0|\ll1$, the current $\Big[S_2\wedge \left(\Tc_\theta\Kc_{\theta'}\right)^{n+1, i}(S_3)-S_2\wedge A(\theta_0, \theta, \theta')\Big]$ is positive closed in $U_j^1$ and its mass can be uniformly bounded with respect to $\theta$ and $\theta'$. Also, there exists $M_A>0$ such that $A(\theta_0, \theta, \theta')\le M _AS_2\wedge\omega_\euc^{s_3+1}$. Let $M_{\rm total}$ be an upper bound of the mass of $\Big[S_2\wedge \left(\Tc_\theta\Kc_{\theta'}\right)^{n+1, i}(S_3)-S_2\wedge A(\theta_0, \theta, \theta')\Big]+ M_AS_2\wedge\omega_\euc^{s_3+1}$ with respect to $\theta$ and $\theta'$. We choose $\delta>0$ small enough to have $\nu_{S_1, j}^n(\delta)<\frac{\varepsilon}{2(M_{\rm total}+1)}$. Let $\chi^{\Delta_{12}}_\delta:U_j^0\times U_j^0(\subset \Xf_{12})\to [0, 1]$ be a smooth function with support in the $\delta$-neighborhood of $\Delta_{12}$ such that $\chi^{\Delta_{12}}_\delta\equiv 1$ near $\Delta_{12}$. We have
	\begin{align}
		\notag&\int_{\Xf_{12}}\chi_j^n\Uc_{1, \eta, \eta'}^{n, n}\wedge \Big(\pi_2^{12}\Big)^*\Big[S_2\wedge\Big(\pi_2^{23}\Big)_*\Big[\left(\chi_k^v\right)^m\Sc_{3, \theta, \theta'}^{n+1, i}\Big]\Big]\\
		&\label{eq:integral_delta}=\int_{\Xf_{12}}\chi^{\Delta_{12}}_\delta\chi_j^n\Uc_{1, \eta, \eta'}^{n, n}\wedge \Big(\pi_2^{12}\Big)^*\Big[S_2\wedge\Big(\pi_2^{23}\Big)_*\Big[\left(\chi_k^v\right)^m\Sc_{3, \theta, \theta'}^{n+1, i}\Big]\Big]\\
		&\label{eq:integral-k}\quad + \int_{\Xf_{12}}\left(1-\chi^{\Delta_{12}}_\delta\right)\chi_j^n\Uc_{1, \eta, \eta'}^{n, n}\wedge \Big(\pi_2^{12}\Big)^*\Big[S_2\wedge\Big(\pi_2^{23}\Big)_*\Big[\left(\chi_k^v\right)^m\Sc_{3, \theta, \theta'}^{n+1, i}\Big]\Big]
	\end{align}
	
	We consider \eqref{eq:integral_delta}. Due to the positivity and negativity of each current in the integrand, we have
	{\begin{align}
			\notag\eqref{eq:integral_delta}&\ge \int_{\Xf_{12}}\chi^{\Delta_{12}}_\delta\chi_j^n\Uc_{1, \eta, \eta'}^{n, n}\wedge \Big(\pi_2^{12}\Big)^*\Big(S_2\wedge \left(\Tc_\theta\Kc_{\theta'}\right)^{n+1, i}(S_3)\Big).\\
			\label{eq:integral_conti1}&\ge \int_{\Xf_{12}}\chi^{\Delta_{12}}_\delta\chi_j^n\Uc_{1, \eta, \eta'}^{n, n}\wedge \Big(\pi_2^{12}\Big)^*\Big[S_2\wedge \left(\Tc_\theta\Kc_{\theta'}\right)^{n+1, i}(S_3)-S_2\wedge A(\theta_0, \theta, \theta')\Big]\\
			\label{eq:integral_conti2}&\quad\quad\quad\quad\quad\quad + \int_{\Xf_{12}}\chi^{\Delta_{12}}_\delta\chi_j^n\Uc_{1, \eta, \eta'}^{n, n}\wedge \Big(\pi_2^{12}\Big)^*\Big(M_AS_2\wedge\omega_\euc^{s_3+1}\Big).
	\end{align}}
	Notice that Condition (I) for $S_2$ and $S_3$ implies the convergence of $\Big[S_2\wedge \left(\Tc_\theta\Kc_{\theta'}\right)^{n+1, i}(S_3)-S_2\wedge A(\theta_0, \theta, \theta')\Big]+M_A\omega_\euc^{s_3+1}$ as $\theta'\to 0$ and then $\theta\to 0$. The location of $\chi_\delta^{\Delta_{12}}$ is different from \eqref{eq:nearDelta23} but the same argument as in \eqref{eq:nearDelta23} proves the convergence of \eqref{eq:integral_delta} and \eqref{eq:integral_conti1}+\eqref{eq:integral_conti2} as $\eta'\to 0$, $\eta\to 0$, $\theta'\to 0$ and then $\theta\to 0$ in this order. Our choice of $\delta$ implies
	\begin{align*}
		0\ge \eqref{eq:integral_delta}\ge \eqref{eq:integral_conti1}+\eqref{eq:integral_conti2}\ge -\frac{\varepsilon}{2}.
	\end{align*}
	
	Next, we consider \eqref{eq:integral-k}. We can write
	\begin{align*}
		\eqref{eq:integral-k}=\int_{\Xf_{23}} \left(\chi_k^v\right)^m\Sc_{3, \theta, \theta'}^{n+1, i}  \wedge \Big(\pi_2^{23}\Big)^*\Big[S_2\wedge \Big(\pi_2^{12}\Big)_*\Big[\left(1-\chi^{\Delta_{12}}_\delta\right)\chi_j^n\Uc_{1, \eta, \eta'}^{n, n}\Big]\Big].
	\end{align*}
	Observe that $1-\chi^{\Delta_{12}}_\delta$ vanishes near $\Delta_{12}$. So, the current $\Big(\pi_2^{12}\Big)_*\Big[\left(1-\chi^{\Delta_{12}}_\delta\right)$ $\chi_j^n\Uc_{1, \eta, \eta'}^{n, n}\Big]$ is smooth and converges to a smooth form $\Big(\pi_2^{12}\Big)_*\Big[\left(1-\chi^{\Delta_{12}}_\delta\right)\chi_j^n$ $u(dd^cu)^i\wedge \Big(\pi_1^{12}\Big)^*\left(\mathbf{1}_{U_j^0}S_1\right)\Big]$ uniformly in $\eta'$ and $\eta$. From Condition (I) for $S_2$ and $S_3$, we have the following convergence:
	\begin{align*}
		&\lim_{\theta\to 0}\lim_{\theta'\to 0}\lim_{\eta\to 0}\lim_{\eta'\to 0}\eqref{eq:integral-k}=\int_{\Xf_{23}}\left(\chi_k^v\right)^m\left\langle (dd^cv)^i\wedge \Big(\pi_2^{23}\Big)^*S_2\wedge \Big(\pi_3^{23}\Big)^*S_3\right\rangle_C\\
		&\quad\quad\quad\quad\quad\quad\quad\wedge \Big(\pi_2^{23}\Big)^*\Big(\pi_2^{12}\Big)_*\Big[\left(1-\chi^{\Delta_{12}}_\delta\right)\chi_j^nu(dd^cu)^i\wedge \Big(\pi_1^{12}\Big)^*\left(\mathbf{1}_{U_j^0}S_1\right)\Big]\Big].
	\end{align*}
	Since Condition (I) for $S_2$ and $S_3$ again implies that the $h$-dimension of the tangent current $\Big(\pi_2^{23}\Big)^*S_2\wedge \Big(\pi_3^{23}\Big)^*S_3$ along $\Delta_{23}$ is minimal and the limit $\left\langle (dd^cv)^i\wedge \Big(\pi_2^{23}\Big)^*S_2\wedge \Big(\pi_3^{23}\Big)^*S_3\right\rangle_C$ has no mass on $\Delta_{23}$. We can choose an $N_\varepsilon\in\N$ such that whenever $k\ge N_\varepsilon$, we have
	\begin{align*}
		\left|\lim_{\theta\to 0}\lim_{\theta'\to 0}\lim_{\eta\to 0}\lim_{\eta'\to 0}\eqref{eq:integral-k}\right|<\varepsilon/2.
	\end{align*}
	Hence, by the definition of limit, our choice of $N_\varepsilon$ completes the proof.
\end{proof}

\begin{lemma}\label{lem:easy_part}
	Let $i\in\{1, \ldots, n-1\}$ and $m\in\N$. Then, we have
	{\begin{align*}
			&\lim_{k\to\infty}\lim_{\theta\to 0}\lim_{\theta'\to 0}\lim_{\eta\to 0}\lim_{\eta'\to 0}\frac{1}{\log k}\int_{\Xf_{12}}\chi_j^n\Uc_{1, \eta, \eta'}^{n, n}\wedge \Big(\pi_2^{12}\Big)^*\Big[S_2\wedge\Big(\pi_2^{23}\Big)_*\Big[ dd^cv_{1/k}^\Kmc\wedge\Sc_{3, \theta, \theta'}^{n, i}\Big]\Big]=0
	\end{align*}}
\end{lemma}

\begin{proof}
	Given $k\in \N$, when $|\theta|, |\theta'|\ll 1/k^2$, we have
	{\begin{align*}
			dd^cv_{1/k}^\Kmc\wedge\Sc_{3, \theta, \theta'}^{n, i}=(dd^cv_\theta^\Tmc)\wedge (dd^cv_{1/k}^\Kmc)\wedge (dd^cu)^{i-1}\wedge \Big(\pi_3^{23}\Big)^*S_3\wedge\omega_{23}^{n-i}.
	\end{align*}}
	Hence, we show the following limit.
	\begin{align*}
		&\lim_{k\to\infty}\lim_{\theta\to 0}\lim_{\eta\to 0}\lim_{\eta'\to 0}\frac{1}{\log k}\int_{\Xf_{12}}\chi_j^n\Uc_{1, \eta, \eta'}^{n, n}\wedge \Big(\pi_2^{12}\Big)^*\Big[S_2\wedge\Big(\pi_2^{23}\Big)_*\Big[(dd^cv_\theta^\Tmc)\wedge \Kmc_{v, 1/k}^i\wedge\Big(\pi_3^{23}\Big)^*S_3\wedge\omega_{23}^{n-i}\Big]\Big]=0.
	\end{align*}
	Compared with Lemma \ref{lem:existence}, this case has different orders in limits, but the same argument proves the limit exists as $\eta'\to 0$, $\eta\to 0$, and $\theta\to 0$ in this order. Then, it is clear that the desired limit converges to $0$ as $k\to\infty$.
\end{proof}

Since $dd^c\chi^v_k:=\frac{1}{\log k}\left(dd^cv_{1/k^2}^\Kmc-dd^cv_{1/k}^\Kmc\right)$, Lemma \ref{lem:easy_part} proves the following corollary:
\begin{corollary}\label{cor:easy_part}
	Let $i\in\{1, \ldots, n-1\}$ and $m\in\N$. Then, we have
	{ \begin{align*}
			&\lim_{k\to\infty}\lim_{\theta\to 0}\lim_{\theta'\to 0}\lim_{\eta\to 0}\lim_{\eta'\to 0}\int_{\Xf_{12}}\chi_j^n\Uc_{1, \eta, \eta'}^{n, n}\wedge \Big(\pi_2^{12}\Big)^*\Big[S_2\wedge\Big(\pi_2^{23}\Big)_*\Big[ dd^c\chi_k^v\wedge\Sc_{3, \theta, \theta'}^{n, i}\Big]\Big]=0
	\end{align*}}
\end{corollary}

\begin{lemma}\label{lem:hard_part}
	Let $i\in\{1, \ldots, n-1\}$ and $m\in\N$. Then, we have
	\begin{align*}
		&\lim_{k\to\infty}\lim_{\theta\to 0}\lim_{\theta'\to 0}\lim_{\eta\to 0}\lim_{\eta'\to 0}\int_{\Xf_{12}}\chi_j^n\Uc_{1, \eta, \eta'}^{n, n}\wedge \Big(\pi_2^{12}\Big)^*\Big[S_2\wedge\Big(\pi_2^{23}\Big)_*\Big[(d\chi_k^v\wedge d^c\chi_k^v)\wedge\Sc_{3, \theta, \theta'}^{n, i}\Big]\Big]=0
	\end{align*}
\end{lemma}

\begin{proof}
	We may assume that $k\gg 1$. Note that $d\chi_k^v\wedge d^c\chi_k^v$ is positive but not closed. However, we have
	\begin{align*}
		d\chi_k^v\wedge d^c\chi_k^v + \chi_k^vdd^c\chi_k^v=\frac{1}{2}dd^c(\chi_k^v)^2,\quad\chi_k^vdd^c\chi_k^v=\frac{\chi_k^v}{\log k}\left(dd^cv_{1/k^2}^\Kmc-dd^cv_{1/k}^\Kmc\right).
	\end{align*}
	So, adding $\chi_k^vdd^c\chi_k^v+\frac{dd^cv^\Kmc_{1/k}}{\log k}$ makes it positive. Namely, 
	\begin{align*}
		\Xi_k:=d\chi_k^v\wedge d^c\chi_k^v + \chi_k^vdd^c\chi_k^v+\frac{dd^cv^\Kmc_{1/k}}{\log k}=\frac{1}{2}dd^c(\chi^v_k)^2+\frac{dd^cv_{1/k}^\Kmc}{\log k}
	\end{align*}
	is a smooth positive closed $(1, 1)$-current on $\C^n\times \C^n$. The following smooth form is closed in $U_j^1$ when $|\theta_0|\ll 1$:	
	$$\Big(\pi_2^{23}\Big)_*\Big[\Xi_k\wedge\left(dd^cv_\theta^\Tmc-dd^c \chi\left(v_\theta^\Tmc-\log|\theta_0|\right)\right)\wedge \Kmc_{v, \theta'}^{i-1}\wedge\Big(\pi_3^{23}\Big)^*S_3\wedge\omega_{23}^{n-i}\Big].$$
	By the same argument as in Lemma \ref{lem:DSHness}, for a constant $M_{\theta_0}>0$ depending on $\theta_0$, the current
	\begin{align*}
		Z_{\theta, \theta'}:=&S_2\wedge\Big(\pi_2^{23}\Big)_*\Big[\Xi_k\wedge\left(dd^cv_\theta^\Tmc-dd^c \chi\left(v_\theta^\Tmc-\log|\theta_0|\right)\right)\wedge \Kmc_{v, \theta'}^{i-1}\wedge\Big(\pi_3^{23}\Big)^*S_3\wedge\omega_{23}^{n-i}\Big] + M_{\theta_0}S_2\wedge\omega_\euc^{s_3+1}
	\end{align*}
	is a positive closed current in $U_j^1$ when $|\theta_0|\ll 1$. We estimate the mass of $Z_{\theta, \theta'}$ over $U_j^n$.
	\begin{align*}
		\|Z_{\theta, \theta'}\|_{U_j^n}&\le \int \chi_j^{n}\omega_\euc^{s_1-1}\wedge S_2\wedge\Big(\pi_2^{23}\Big)_*\Big[\Xi_k\wedge\left(dd^cv_\theta^\Tmc-dd^c \chi\left(v_\theta^\Tmc-\log|\theta_0|\right)\right)\\
		&\quad\quad\quad\quad\quad\quad\quad\quad\quad\quad\quad\wedge \Kmc_{v, \theta'}^{i-1}\wedge\Big(\pi_3^{23}\Big)^*S_3\wedge\omega_{23}^{n-i}\Big]+M_{\theta_0}\int \chi_j^{n}\omega_\euc^{n-s_2}\wedge S_2 \\
		&=\int \Big(\pi_2^{23}\Big)_*\left(\chi_j^{n}\omega_\euc^{s_1-1}\wedge S_2\right)\wedge\Big[\Xi_k\wedge\left(dd^cv_\theta^\Tmc-dd^c \chi\left(v_\theta^\Tmc-\log|\theta_0|\right)\right)\\
		&\quad\quad\quad\quad\quad\quad\quad\quad\quad\quad\quad\wedge \Kmc_{v, \theta'}^{i-1}\wedge\Big(\pi_3^{23}\Big)^*S_3\wedge\omega_{23}^{n-i}\Big]+M_{\theta_0}\int \chi_j^{n}\omega_\euc^{n-s_2}\wedge S_2\\
		&=\int \Big(\pi_2^{23}\Big)_*\left(dd^c\chi_j^{n}\wedge \omega_\euc^{s_1-1}\wedge S_2\right)\wedge\Big[\left(\frac{1}{2}(\chi^v_k)^2+\frac{v^\Kmc_{1/k}}{\log k}\right)\left(dd^cv_\theta^\Tmc-dd^c \chi\left(v_\theta^\Tmc-\log|\theta_0|\right)\right)\\
		&\quad\quad\quad\quad\quad\quad\quad\quad\quad\quad\quad\quad\quad\quad\quad\quad\quad\wedge \Kmc_{v, \theta'}^{i-1}\wedge\Big(\pi_3^{23}\Big)^*S_3\wedge\omega_{23}^{n-i}\Big]+M_{\theta_0}\int \chi_j^{n}\omega_\euc^{n-s_2}\wedge S_2.
	\end{align*}
	Since $\left(\frac{1}{2}(\chi^v_k)^2+\frac{v^\Kmc_{1/k}}{\log k}\right)$ is a bounded function, Condition (I) for $S_2$ and $S_3$ implies that the mass $\|Z_{\theta, \theta'}\|$ is uniformly bounded.\medskip
	
	In place of $\chi_j^{n}\omega_\euc^{s_1-1}$, we put a smooth test form $\varphi$ of bidegree $(s_1-1, s_1-1)$, Lemma \ref{lem:induction} and Condition (I) for $S_2$ and $S_3$ imply the convergence of $Z_{\theta, \theta'}-M_{\theta_0}S_2\wedge\omega_\euc^{s_3+1}$ as $\theta'\to 0$ and then $\theta\to 0$ in the sense of currents. Furthermore, as $k\to\infty$, we see that the support $\frac{1}{2}(\chi^v_k)^2$ shrinks to $\Delta_{23}$ and $\frac{v^\Kmc_{1/k}}{\log k}$ converges uniformly to $0$. Condition (I) for $S_2$ and $S_3$ implies that the current $\left\langle (dd^cu)^i\wedge \Big(\pi_2^{23}\Big)^*S_2\wedge\Big(\pi_3^{23}\Big)^*S_3 \right\rangle_C$ has no mass on $\Delta_{23}$. Therefore, we have $\displaystyle \lim_{k\to\infty}\lim_{\theta\to 0}\lim_{\theta'\to 0}\left(Z_{\theta, \theta'}-M_{\theta_0}S_2\wedge\omega_\euc^{s_3+1}\right)=0$ in the sense of $*$-topology on $U_j^n$.\medskip
	
	Concerning the limit to be estimated, we have	
	{\begin{align}
			\notag&\int_{\Xf_{12}}\chi_j^n\Uc_{1, \eta, \eta'}^{n, n}\wedge \Big(\pi_2^{12}\Big)^*\Big[S_2\wedge\Big(\pi_2^{23}\Big)_*\Big[(d\chi_k^v\wedge d^c\chi_k^v)\wedge\Sc_{3, \theta, \theta'}^{n, i}\Big]\Big]\\
			\label{eq:4.7-1}&=\int_{\Xf_{12}}\chi_j^n\Uc_{1, \eta, \eta'}^{n, n}\wedge \Big(\pi_2^{12}\Big)^*Z_{\theta, \theta'}-M_{\theta_0}\int_{\Xf_{12}}\chi_j^n\Uc_{1, \eta, \eta'}^{n, n}\wedge \Big(\pi_2^{12}\Big)^*(S_2\wedge\omega_\euc^{s_3+1})\\
			\label{eq:4.7-2}&\quad+\int_{\Xf_{12}}\chi_j^n\Uc_{1, \eta, \eta'}^{n, n}\wedge \Big(\pi_2^{12}\Big)^*\Big[S_2\wedge\Big(\pi_2^{23}\Big)_*\Big[\Xi_k\wedge dd^c \chi\left(v_\theta^\Tmc-\log|\theta_0|\right)\wedge\Kmc_{v, \theta'}^{i-1}\wedge\Big(\pi_3^{23}\Big)^*S_3\wedge\omega_{23}^{n-i}\Big]\Big]\\
			\label{eq:4.7-3}&\quad-\int_{\Xf_{12}}\chi_j^n\Uc_{1, \eta, \eta'}^{n, n}\wedge \Big(\pi_2^{12}\Big)^*\left( S_2\wedge\Big(\pi_2^{23}\Big)_*\left[\left(\chi_k^vdd^c\chi_k^v+\frac{dd^cv_{1/k}}{\log k}\right)\wedge \Sc_{3, \theta, \theta'}^{n,i}\right]\right)
	\end{align}}
	The continuity of local potential functionals of $S_1$ implies Condition (I) with every positive closed current. Together with Lemma \ref{lem:induction}, Condition (I) induces the convergence of \eqref{eq:4.7-1} as $\eta'\to 0$, $\eta\to 0$ in this order. Since $\displaystyle\lim_{k\to\infty}\lim_{\theta\to 0}\lim_{\theta'\to 0}\left(Z_{\theta, \theta'}-M_{\theta_0}S_2\wedge\omega_\euc^{s_3+1}\right)=0$ in the sense of $*$-topology on $U_j^n$, the continuity of local potential functionals of $S_1$ says \eqref{eq:4.7-1} converges to $0$ as we take all the limits in the right order.\smallskip
	
	For \eqref{eq:4.7-2}, the support of $d\chi_k^v\wedge d^c\chi_k^v + \chi_k^v dd^c\chi_k^v$ uniformly converges to $\Delta_{23}$ but the support of $dd^c \chi\left(v_\theta^\Tmc-\log|\theta_0|\right)$ is away from $\Delta_{23}$. Hence, for all sufficiently large $k\in\N$ and for $\theta'\in\C^*$ with $|\theta'|\ll 1/k$, we have
	{\small \begin{align*}
		\eqref{eq:4.7-2}&=\int_{\Xf_{12}}\chi_j^n\Uc_{1, \eta, \eta'}^{n, n}\wedge \Big(\pi_2^{12}\Big)^*\bigg[S_2\wedge\Big(\pi_2^{23}\Big)_*\bigg[\frac{dd^cv_{1/k}}{\log k}\wedge dd^c \chi\left(v_\theta^\Tmc-\log|\theta_0|\right)\wedge\Kmc_{v, \theta'}^{i-1}\wedge\Big(\pi_3^{23}\Big)^*S_3\wedge\omega_{23}^{n-i}\bigg]\bigg]\\
		&=\frac{1}{\log k}\int_{\Xf_{12}}\chi_j^n\Uc_{1, \eta, \eta'}^{n, n}\wedge \Big(\pi_2^{12}\Big)^*\Big[S_2\wedge\Big(\pi_2^{23}\Big)_*\Big[dd^c \chi\left(v_\theta^\Tmc-\log|\theta_0|\right)\wedge\Kmc_{v, 1/k}^{i}\wedge\Big(\pi_3^{23}\Big)^*S_3\wedge\omega_{23}^{n-i}\Big]\Big].
	\end{align*}}
	Since $\Big(\pi_2^{23}\Big)_*\Big[dd^c \chi\left(v_\theta^\Tmc-\log|\theta_0|\right)\wedge\Kmc_{v, 1/k}^{i}\wedge\Big(\pi_3^{23}\Big)^*S_3\wedge\omega_{23}^{n-i}\Big]$ is a smooth form uniformly bounded with respect to $\theta$, $\theta'$ and $k$. Hence, the continuity of local potential functionals of $S_1$ yields the convergence of \eqref{eq:4.7-2} to $0$.\smallskip
	
	Lemma \ref{lem:easy_part} and Corollary \ref{cor:easy_part} imply the convergence of \eqref{eq:4.7-3} to $0$. Hence, summing up, we obtain the desired convergence to $0$.
\end{proof}

Now, we are ready to prove Claim 1.
\begin{proof}[Proof of Claim 1]
	For a smooth test function $\varphi$ on $U_j^n\times U_j^n\subset \Xf_{23}$, we estimate
	{\begin{align*}
			\lim_{k\to\infty}\lim_{\theta\to 0}\int_{\Xf_{23}}	\varphi\chi_k^v(dd^cv_\theta^\Tmc)\wedge \Big\langle (dd^cv)^{i-1}\wedge \Big(\pi_2^{23}\Big)^*\Big((S_1\wedge S_2)_{DS}\Big)\wedge\Big(\pi_3^{23}\Big)^*S_3\Big\rangle_C \wedge \omega_{23}^{n-i}
	\end{align*}}
	
	By Definition \ref{def:prod_C} and \cite[Theorem 6.10]{Ahn25}, we have
	{\begin{align*}
			&\int_{\Xf_{23}}\varphi\chi_k^v(dd^cv_\theta^\Tmc)\wedge \Big\langle (dd^cv)^{i-1}\wedge \Big(\pi_2^{23}\Big)^*\Big((S_1\wedge S_2)_{DS}\Big)\wedge\Big(\pi_3^{23}\Big)^*S_3\Big\rangle_C \wedge\omega_{23}^{n-i}\\
			&=\lim_{\theta'\to 0}\int_{\Xf_{23}}\varphi\chi_k^v(dd^cv_\theta^\Tmc)\wedge \Kmc_{v, \theta'}^{i-1}\wedge \Big(\pi_2^{23}\Big)^*\Big((S_1\wedge S_2)_{DS}\Big)\wedge\Big(\pi_3^{23}\Big)^*S_3\wedge\omega_{23}^{n-i}\\
			&=\lim_{\theta'\to 0}\lim_{\eta\to 0}\int_{X_2} \Big(\pi_2^{12}\Big)_*\Big[\Kmc_{u, \eta}^n\wedge \Big(\pi_1^{12}\Big)^*S_1\Big]\wedge S_2\wedge\Big(\pi_2^{23}\Big)_*\left[\varphi\chi_k^v(dd^cv_\theta^\Tmc)\wedge \Kmc_{v, \theta'}^{i-1}\wedge\Big(\pi_3^{23}\Big)^*S_3\wedge\omega_{23}^{n-i}\right]\\
			&=\lim_{\theta'\to 0}\lim_{\eta\to 0}\lim_{\eta'\to 0}\int_{\Xf_{12}} \Uc_{1, \eta, \eta'}^{n, n}\wedge \Big(\pi_2^{12}\Big)^*\bigg[S_2\wedge\Big(\pi_2^{23}\Big)_*\left[dd^c(\varphi\chi_k^v)\wedge\Sc_{3, \theta, \theta'}^{n, i}\right]\bigg].
	\end{align*}}
	We can write
	\begin{align}
		\label{eq:Prop_easy_part}&\Uc_{1, \eta, \eta'}^{n, n}\wedge \Big(\pi_2^{12}\Big)^*\bigg[S_2\wedge\Big(\pi_2^{23}\Big)_*\left[dd^c(\varphi\chi_k^v)\wedge\Sc_{3, \theta, \theta'}^{n, i}\right]\bigg]\\
		\label{eq:est1}&= \Uc^{n, n}_{1, \eta, \eta'}\wedge \Big(\pi_2^{12}\Big)^*\left(S_2\wedge\Big(\pi_2^{23}\Big)_*\left(\varphi dd^c\chi_k^v\wedge\Sc^{n, i}_{3, \theta, \theta'}\right)\right)\\
		\label{eq:est2}&+\Uc^{n, n}_{1, \eta, \eta'}\wedge \Big(\pi_2^{12}\Big)^*\left(S_2\wedge\Big(\pi_2^{23}\Big)_*\left(\chi_k^vdd^c\varphi\wedge\Sc^{n, i}_{3, \theta, \theta'}\right)\right)\\
		\label{eq:est3}&+\Uc^{n, n}_{1, \eta, \eta'}\wedge \Big(\pi_2^{12}\Big)^*\left(S_2\wedge\Big(\pi_2^{23}\Big)_*\left(\left(d\varphi\wedge d^c\chi_k^v+d\chi_k^v\wedge d^c\varphi\right)\wedge\Sc^{n, i}_{3, \theta, \theta'}\right)\right).
	\end{align}
	
	Lemma \ref{lem:easy_part} proves the convergence of \eqref{eq:est1} to $0$. Lemma \ref{lem:vanish} proves the convergence of \eqref{eq:est2} to $0$. For \eqref{eq:est3}, the positivity or negativity of each current in the integrand except $\left(d\varphi\wedge d^c\chi_k^v+d\chi_k^v\wedge d^c\varphi\right)$, the Cauchy-Schwarz inequality and Lemmas \ref{lem:existence} and \ref{lem:hard_part} complete the proof.
\end{proof}

\subsection{Proof of Claim 2} Let $\phi$ be a smooth function with compact support in $\xi_j^n$ and $\Phi=\chi_j^n\big(\pi_2^{12}\big)^*\phi$, which is another smooth function with compact support in $\Uf_j^n$ such that $\Phi=\big(\pi_2^{12}\big)^*\phi$ near $\Delta_{12}$. We have $$\supp \Big(dd^c \Big(\chi_j^n\big(\pi_2^{12}\big)^*\phi\Big)- \Big(\chi_j^n\big(\pi_2^{12}\big)^*(dd^c\phi)\Big)\Big)\cap \Delta_{12}=\emptyset.$$ Since $S_1$ admits continuous local potential functionals, by \cite[Theorem 1.1]{Ahn25}, we have
\begin{align*}
	&\big\langle\big(S_1\wedge \big(S_2\wedge S_3\big)_{DS}\big)_{DS}, \phi\big\rangle\\
	&=\int_{\Xf_{12}} u(dd^c u)^{n-1}\wedge \big(\pi_1^{12}\big)^*S_1\wedge \big(\pi_2^{12}\big)^*(S_2\wedge S_3)_{DS}\wedge dd^c\Phi\\
	&=\int_{\Xf_{12}} \chi_j^nu(dd^c u)^{n-1}\wedge  \big(\pi_1^{12}\big)^*S_1\wedge \big(\pi_2^{12}\big)^*\left((S_2\wedge S_3)_{DS}\wedge dd^c\phi\right)\\
	&\quad+\int_{\Xf_{12}} u(dd^c u)^{n-1}\wedge \left(dd^c \left(\chi_j^n\big(\pi_2^{12}\big)^*\phi\right)- \left(\chi_j^n\big(\pi_2^{12}\big)^*(dd^c\phi)\right)\right)\wedge\big(\pi_1^{12}\big)^*S_1\wedge \big(\pi_2^{12}\big)^*(S_2\wedge S_3)_{DS}\\
	&=\lim_{\theta\to 0}\int_{\Xf_{12}} \chi_j^nu(dd^c u)^{n-1}\wedge  \big(\pi_1^{12}\big)^*S_1\wedge \big(\pi_2^{12}\big)^*\Big[(dd^c\phi\wedge S_2) \wedge \big(\pi_2^{23}\big)_*\Big[\big(\pi_3^{23}\big)^*S_3\wedge \Kmc_{v, \theta}^n\Big] \Big]\\
	&\quad+\lim_{\theta\to 0}\int_{\Xf_{12}} u(dd^c u)^{n-1}\wedge \left(dd^c \left(\chi_j^n\big(\pi_2^{12}\big)^*\phi\right)- \left(\chi_j^n\big(\pi_2^{12}\big)^*(dd^c\phi)\right)\right)\\
	&\quad\quad\quad\quad\quad\quad\quad\quad\quad\quad\wedge\big(\pi_1^{12}\big)^*S_1\wedge \big(\pi_2^{12}\big)^*\Big[S_2 \wedge \big(\pi_2^{23}\big)_*\Big[\big(\pi_3^{23}\big)^*S_3\wedge\Kmc_{v, \theta}^{n}\Big] \Big]\\
	&=\lim_{\theta\to 0}\lim_{\eta\to 0}\int_{\Xf_{12}}u_\eta^\Kmc(dd^cu)^{n-1}\wedge dd^c\Phi\wedge \big(\pi_1^{12}\big)^*S_1\wedge \big(\pi_2^{12}\big)^*S_2\wedge \big(\pi_2^{12}\big)^*\big(\pi_2^{23}\big)_*\Big[\big(\pi_3^{23}\big)^*S_3\wedge\Kmc_{v, \theta}^{n}\Big]\\
	&=\lim_{\theta\to 0}\int_{X_2} \phi\big(S_1\wedge S_2\big)_{DS}\wedge \big(\pi_2^{23}\big)_*\Big[\big(\pi_3^{23}\big)^*S_3\wedge \Kmc_{v, \theta}^{n}\Big]=\big\langle\big(\big(S_1\wedge S_2\big)_{DS}\wedge S_3\big)_{DS}, \varphi\big\rangle.
\end{align*}
Note that by the last object, we mean the shadow of the tangent currents of $\big(\pi_2^{23}\big)^*\big(S_1\wedge S_2\big)_{DS}\wedge \big(\pi_3^{23}\big)^*S_3$. The third to last equality comes from the continuity of local potential functionals of $S_1$. The second to last equality is from Proposition \ref{prop:uniform_classical_prod}.
The claim is proved. In particular, as discussed previously, $\big(\big(S_1\wedge S_2\big)_{DS}\wedge S_3\big)_{DS}$ exists and $\big(S_1\wedge \big(S_2\wedge S_3\big)_{DS}\big)_{DS}=\big(\big(S_1\wedge S_2\big)_{DS}\wedge S_3\big)_{DS}$.\hspace*{\fill} $\qed$
\medskip

\begin{proposition}\label{prop:2nd_part}
	Under the assumptions of Theorem \ref{thm:associativity}, for each $j\in J$, the Dinh-Sibony product $\big(S_1\wedge S_2\wedge S_3\big)_{DS}$ of $S_1$, $S_2$ and $S_3$ is well defined on $U_j^n$ and we have
	\begin{align*}
		\left(S_1\wedge (S_2\wedge S_3)_{DS}\right)_{DS}=\left(S_1\wedge S_2\wedge S_3\right)_{DS}\textrm{ on }U_j^n.
	\end{align*}
\end{proposition}

As previously, we consider $U_j^0\times U_j^0 \times U_j^0\subset \Xf$ and $(x_1, x_2, x_3)$ the coordinates for $U_j^0\times U_j^0 \times U_j^0$. We use the coordinates $(x_2, x_2, x_2)$ for the diagonal submanifold $\Delta$ of $U_j^0\times U_j^0 \times U_j^0$. The normal bundle $E$ of $\Delta$ in $U_j^0\times U_j^0 \times U_j^0$ can be written as $E:=\C^n\times U_j^0\times \C^n$ with the projection map $\pi: E \to \Delta$ defined by $\pi(z, x, w)=(x, x, x)$. We may consider $U_j^0\times U_j^0 \times U_j^0$ a subset of $E$ with the embedding $\tau(x_1, x_2, x_3)=(z :=x_1-x_2, x:=x_2, w:=x_3-x_2)$, which is the holomorphic admissible map. With respect to the coordinates $(z, x, w)$, we can write $\pi_1(z, x, w)=z+x$, $\pi_2(z, x, w)=x$, $\pi_3(z, x, w)=w+x$. We compactify $E$ by adding the hyperplane at infinity and denote it by $\overline{E}:=\P(E\oplus \C)$. Then, on each fiber of $\overline{E}$, we have the natural homogeneous coordinates $[z:w:t]\in\P^{2n}$ and each point in $\overline{E}$ can be written as $(x, [z:w:t])$ and $U_j^0\times U_j^0 \times U_j^0$ may be identified with $(x_2, [x_1-x_2: x_3-x_2: 1])$. The action $A_\lambda$ of multiplication by $\lambda$ on fibers of $E$ extends to $\overline{E}$ and can be written as $A_\lambda(x, [z: w: t])=(x, [\lambda z: \lambda w: t])$.\medskip

We use the following K\"ahler forms $\omega$ on $E$ and $\omega_F$ on the fiber space $\C^{2n}$ of $E$: $\omega = dd^c|x|^2 + \frac{1}{2}dd^c\log(1+|z|^2)+\frac{1}{2}dd^c\log(1+|w|^2)$ and $\omega_F=\frac{1}{2}dd^c\log(1+|z|^2)+\frac{1}{2}dd^c\log(1+|w|^2)$, which are different from the one in \cite{Ahn25}. Recall the definitions of $u$, $v$ and their related functions and forms were introduced right after Theorem \ref{thm:associativity}.\medskip

For notational convenience, we write $\pi_1=\pi_1^{12}\circ\pi_{12}$, $\pi_2=\pi_2^{12}\circ\pi_{12}=\pi_2^{23}\circ\pi_{23}$, $\pi_3=\pi_3^{23}\circ\pi_{23}$ and 
$$T:=\pi_1^* S_1\wedge \pi_2^*S_2\wedge \pi_3^*S_3.$$
\begin{proposition}
	\label{prop:existence_tan_cur_S123} Let $j\in J$. Under the assumptions of Theorem \ref{thm:associativity}, tangent currents of $T$ along $\Delta$ exist and its $h$-dimension is minimal in $U_j^n$.
\end{proposition}

\begin{proof} Recall that we assumed that $s_1+s_2+s_3=n$. As previously, it suffices to prove the following claims, which will be proved in the following subsections.\medskip
	
	\noindent{\bf Claim 3.} Let $\varphi$ be a positive smooth test function on $U_j^n\times U_j^n\times U_j^n$. For the function
	\begin{align*}
		I''_{k, l}(m, \lambda, \eta)=&\int \varphi T\wedge\pi_{12}^*\left(\chi_m^udd^cu_{1/\lambda}^\Tmc\wedge \Kmc_{u,\eta}^{k-1}\right)\wedge\pi_{23}^*\left(\chi^v_mdd^cv_{1/\lambda}^\Tmc\wedge \Kmc_{v,\eta}^{l-1}\right)\wedge \pi_2^*\omega_\euc^{2n-k-l},
	\end{align*}
	the limit $\displaystyle \lim_{\lambda\to\infty}\lim_{\eta\to 0}I''_{k,l}(m, \lambda, \eta)$ exists for $1\ll m$, $0<k, l$ and $k+l=n, \ldots, 2n$, and we have $\displaystyle \lim_{m\to\infty}\lim_{\lambda\to\infty}\lim_{\eta\to 0}I''_{k,l}(m, \lambda, \eta)=0$ for $0<k, l$ and $k+l=n, \ldots, 2n-1$.\medskip
	
	\noindent{\bf Claim 4. }The shadow of the tangent currents of $T$ equals $\big(S_1\wedge \big(S_2\wedge S_3\big)_{DS}\big)_{DS}$.\medskip
	
	Assuming the assumptions, the proof works as follows:
	\medskip
	
	First, for the existence of the tangent currents, it suffices to show that $\left((A_\lambda)_*T\right)_{|\lambda|\gg 1}$ has locally uniformly bounded mass. Since $\omega$ is K\"ahler on $E$ (but not on $\overline{E}$), it suffices to show that $\left((A_\lambda)_*T\wedge \omega^{2n}\right)_{|\lambda|\gg 1}$ is locally uniformly bounded on $\pi^{-1}(U_j^n)$.\medskip
	
	For a positive smooth test function $\psi$ on $E$, by the change of coordinates in the fiber direction, we have
	{\begin{align*}
			&\int_E\psi(A_\lambda)_*T\wedge \omega^{2n}=\int_{U_j^0\times U_j^0\times U_j^0}(\psi\circ A_\lambda) T\\
			&\quad\quad\quad\wedge\left(dd^c|x|^2 + \frac{1}{2}dd^c\log(1/|\lambda|^2+|z|^2)+\frac{1}{2}dd^c\log(1/|\lambda|^2+|w|^2)\right)^{2n}\\
			&=\int_{U_j^n\times U_j^n\times U_j^n}(\psi\circ A_\lambda)T\wedge\left(\pi_2^*\omega_\euc + \pi_{12}^*dd^cu_{1/\lambda}^\Tmc+\pi_{23}^*dd^cv_{1/\lambda}^\Tmc\right)^{2n}\\
			&=\int_{U_j^n\times U_j^n\times U_j^n}\sum_{\substack{0\le k\le n\\0\le l\le n}}c_{kl}\Big[(\psi\circ A_\lambda)T\wedge\pi_{12}^*\left(dd^cu_{1/\lambda}^\Tmc\right)^k\wedge\pi_{23}^*\left(dd^cv_{1/\lambda}^\Tmc\right)^l\wedge \pi_2^*\omega_\euc^{2n-k-l}\Big].
	\end{align*}}For a bidegree reason, we only need to check the case $n\le k+l$. By \cite[Lemma 2.16 and Theorem 2.15]{Ahn25}, it suffices to show that the mass of the current
	\begin{align*}
		\varphi T\wedge\pi_{12}^*\left(dd^cu_{1/\lambda}^\Tmc\wedge \Kmc_{u,\eta}^{k-1}\right)\wedge\pi_{23}^*\left(dd^cv_{1/\lambda}^\Tmc\wedge \Kmc_{v,\eta}^{l-1}\right)\wedge \pi_2^*\omega_\euc^{2n-k-l}
	\end{align*}
	is uniformly bounded, where $|\eta|\ll 1$, $k+l=n, \ldots, 2n$, and $\varphi$ is a positive smooth test function on $U_j^n \times U_j^n \times U_j^n$. Then, for the minimality of the $h$-dimension, we show that the limit currents of the family has no mass on $\Delta$ when $k+l=n, \ldots, 2n-1$.\medskip
	
	When $k=0$ or $l=0$, the local uniform boundedness of mass is obvious as implied by the continuity assumption on $S_1$ and Condition (I) for $S_2$ and $S_3$. So, we assume that both $k, l >0$. Together with Claims 3 and 4, the proof goes in the same way as in the proof of Proposition \ref{prop:1st_part}.
\end{proof}

\subsection{Proof of Claim 3} We consider two cases: $l<n$ and $l=n$.
\subsubsection{$l<n$} Let $\phi$ be a positive smooth test function on $U_j^n$ such that $\pi_2^*\phi\ge \varphi$. Then, for all sufficiently large $m\in \N$, we have
{\begin{align}
		\notag	&\int_{\Xf} \varphi T\wedge\pi_{12}^*\left(\chi_m^udd^cu_{1/\lambda}^\Tmc\wedge \Kmc_{u,\eta}^{k-1}\right)\wedge\pi_{23}^*\left(\chi^v_mdd^cv_{1/\lambda}^\Tmc\wedge \Kmc_{v,\eta}^{l-1}\right)\wedge \pi_2^*\omega_\euc^{2n-k-l}\\
		\notag	&\le \int_\Xf \pi_1^*S_1\wedge \pi_2^*\left(\phi S_2\wedge \omega_\euc^{2n-k-l}\right)\wedge \pi_3^*S_3\wedge\pi_{12}^*\left(\chi_m^udd^cu_{1/\lambda}^\Tmc\wedge \Kmc_{u,\eta}^{k-1}\right)\wedge\pi_{23}^*\left(\chi^v_mdd^cv_{1/\lambda}^\Tmc\wedge \Kmc_{v,\eta}^{l-1}\right)\\
		\notag	&=\int_{X_2}  \left(\phi S_2\wedge \omega_\euc^{2n-k-l}\right)\wedge (\pi_2)_*\left(\pi_1^*S_1\wedge\pi_{12}^*\left(\chi_m^udd^cu_{1/\lambda}^\Tmc\wedge \Kmc_{u,\eta}^{k-1}\right)\right)\\
		\notag	&\quad\quad\quad\quad\quad\quad\quad\quad\quad\quad\quad\quad\quad\quad\quad\wedge(\pi_2)_*\left(\pi_3^*S_3\wedge\pi_{23}^*\left(\chi^v_mdd^cv_{1/\lambda}^\Tmc\wedge \Kmc_{v,\eta}^{l-1}\right)\right)\\
		\notag	&=\int_{X_2}  \left(\phi S_2\wedge \omega_\euc^{2n-k-l}\right)\wedge (\pi^{12}_2)_*\left(\left(\pi_1^{12}\right)^*S_1\wedge\left(\chi_m^udd^cu_{1/\lambda}^\Tmc\wedge \Kmc_{u,\eta}^{k-1}\right)\right)\\
		\notag	&\quad\quad\quad\quad\quad\quad\quad\quad\quad\quad\quad\quad\quad\wedge\left(\pi_2^{23}\right)_*\left(\left(\pi_3^{23}\right)^*S_3\wedge\left(\chi^v_mdd^cv_{1/\lambda}^\Tmc\wedge \Kmc_{v,\eta}^{l-1}\right)\right)\\
		\notag	&=\int_{\Xf_{12}} \chi_m^udd^cu_{1/\lambda}^\Tmc\wedge \Kmc_{u,\eta}^{k-1}\wedge \left(\pi_1^{12}\right)^*S_1\wedge (\pi^{12}_2)^* \left(\pi_2^{23}\right)_*\Big[\left(\chi^v_m\left(\pi_2^{23}\right)^*\phi\right)\\
		\notag	&\quad\quad\quad\quad\quad\quad\quad\quad\quad\quad\left(\pi_2^{23}\right)^*\left( S_2\wedge \omega_\euc^{2n-k-l}\right)\wedge\left(\pi_3^{23}\right)^*S_3\wedge\left(dd^cv_{1/\lambda}^\Tmc\wedge \Kmc_{v,\eta}^{l-1}\right)\Big]\\
		\notag	&\le\int_{\Xf_{12}} \chi_j^n dd^cu_{1/\lambda}^\Tmc\wedge \Kmc_{u,\eta}^{k-1}\wedge \left(\pi_1^{12}\right)^*S_1\wedge (\pi^{12}_2)^* \left(\pi_2^{23}\right)_*\Big[\left(\chi^v_m\left(\pi_2^{23}\right)^*\phi\right)\\
		\notag	&\quad\quad\quad\quad\quad\quad\quad\quad\quad\quad\left(\pi_2^{23}\right)^*\left( S_2\wedge \omega_\euc^{2n-k-l}\right)\wedge\left(\pi_3^{23}\right)^*S_3\wedge\left(dd^cv_{1/\lambda}^\Tmc\wedge \Kmc_{v,\eta}^{l-1}\right)\Big]	\\
		\label{eq:claim4}	&=\int_{\Xf_{12}}u_{1/\lambda}^\Tmc dd^c \Big[\chi_j^n  \Kmc_{u,\eta}^{k-1}\wedge \left(\pi_1^{12}\right)^*S_1\wedge (\pi^{12}_2)^* \left(\pi_2^{23}\right)_*\Big[\left(\chi^v_m\left(\pi_2^{23}\right)^*\phi\right)\\
		\notag	&\quad\quad\quad\quad\quad\quad\quad\quad\quad\quad\left(\pi_2^{23}\right)^*\left( S_2\wedge \omega_\euc^{2n-k-l}\right)\wedge\left(\pi_3^{23}\right)^*S_3\wedge\left(dd^cv_{1/\lambda}^\Tmc\wedge \Kmc_{v,\eta}^{l-1}\right)\Big]\Big]
\end{align}}
All the integrals in \eqref{eq:claim4} other than
\begin{align}
	\label{eq:417}&\int_{\Xf_{12}}\chi_j^n u_{1/\lambda}^\Tmc  \wedge \Kmc_{u,\eta}^{k-1}\wedge \left(\pi_1^{12}\right)^*S_1\wedge (\pi^{12}_2)^* \left(\pi_2^{23}\right)_*\Big[dd^c\left(\chi^v_m\left(\pi_2^{23}\right)^*\phi\right)\\
	\notag	&\quad\quad\quad\quad\quad\quad\left(\pi_2^{23}\right)^*\left( S_2\wedge \omega_\euc^{2n-k-l}\right)\wedge\left(\pi_3^{23}\right)^*S_3\wedge\left(dd^cv_{1/\lambda}^\Tmc\wedge \Kmc_{v,\eta}^{l-1}\right)\Big]
\end{align}
contain a derivative of $\chi_j^n$ and so the integral can be expressed as an integral of the form $ \left(\pi_2^{23}\right)^* S_2 \wedge \left(\pi_3^{23}\right)^* S_3 \wedge \left(dd^cv_{1/\lambda}^\Tmc\wedge \Kmc_{v,\eta}^{l-1}\right) \wedge$ (smooth form) on $\Xf_{23}$. The existence of its limit comes from Condition (I) for $S_2$ and $S_3$. Also, Condition (I) for $S_2$ and $S_3$ implies no mass concentration of mass on $\Delta_{23}$ for its limit currents. For \eqref{eq:417}, we apply the same idea as \eqref{eq:Prop_easy_part} in the proof of Proposition \ref{prop:1st_part}. The estimates in the proof of Proposition \ref{prop:1st_part} are uniform and so, the diagonal type limit, corresponding to the limit of \eqref{eq:417}, exists and is the same as iterated limit, which corresponds to \eqref{eq:Prop_easy_part}. Namely, the above limit equals $0$. The existence part also comes in the same way.

\subsubsection{$l=n$} First, we consider the uniform boundedness of mass. We assume $k\le n$. Let $\phi$ be a positive smooth test function on $U_j^n$ such that $\pi_2^*\phi\ge \varphi$. Then, for all sufficiently large $m\in \N$, we have
{\begin{align}
		\notag	&\int_{\Xf} \varphi T\wedge\pi_{12}^*\left(\chi_m^udd^cu_{1/\lambda}^\Tmc\wedge \Kmc_{u,\eta}^{k-1}\right)\wedge\pi_{23}^*\left(\chi^v_mdd^cv_{1/\lambda}^\Tmc\wedge \Kmc_{v,\eta}^{n-1}\right)\wedge \pi_2^*\omega_\euc^{n-k}\\
		\notag	&\le \int_\Xf \pi_1^*S_1\wedge \pi_2^*\left(\phi S_2\wedge \omega_\euc^{n-k}\right)\wedge \pi_3^*S_3\wedge\pi_{12}^*\left(\chi_m^udd^cu_{1/\lambda}^\Tmc\wedge \Kmc_{u,\eta}^{k-1}\right)\wedge\pi_{23}^*\left(\chi^v_mdd^cv_{1/\lambda}^\Tmc\wedge \Kmc_{v,\eta}^{n-1}\right)\\
		\notag	&=\int_{X_2}  \left(\phi S_2\wedge \omega_\euc^{n-k}\right)\wedge (\pi_2)_*\left(\pi_1^*S_1\wedge\pi_{12}^*\left(\chi_m^udd^cu_{1/\lambda}^\Tmc\wedge \Kmc_{u,\eta}^{k-1}\right)\right)\\
		\notag	&\quad\quad\quad\quad\quad\quad\quad\quad\quad\quad\quad\quad\quad\quad\quad\wedge(\pi_2)_*\left(\pi_3^*S_3\wedge\pi_{23}^*\left(\chi^v_mdd^cv_{1/\lambda}^\Tmc\wedge \Kmc_{v,\eta}^{n-1}\right)\right)\\
		\notag	&=\int_{X_2}  \left(\phi S_2\wedge \omega_\euc^{n-k}\right)\wedge (\pi^{12}_2)_*\left(\left(\pi_1^{12}\right)^*S_1\wedge\left(\chi_m^udd^cu_{1/\lambda}^\Tmc\wedge \Kmc_{u,\eta}^{k-1}\right)\right)\\
		\notag	&\quad\quad\quad\quad\quad\quad\quad\quad\quad\quad\quad\quad\quad\wedge\left(\pi_2^{23}\right)_*\left(\left(\pi_3^{23}\right)^*S_3\wedge\left(\chi^v_mdd^cv_{1/\lambda}^\Tmc\wedge \Kmc_{v,\eta}^{n-1}\right)\right)\\
		\notag	&=\int_{\Xf_{12}} \chi_m^udd^cu_{1/\lambda}^\Tmc\wedge \Kmc_{u,\eta}^{k-1}\wedge \left(\pi_1^{12}\right)^*S_1\wedge (\pi^{12}_2)^* \left(\pi_2^{23}\right)_*\Big[\left(\chi^v_m\left(\pi_2^{23}\right)^*\phi\right)\\
		\notag	&\quad\quad\quad\quad\quad\quad\quad\quad\quad\quad\left(\pi_2^{23}\right)^*\left( S_2\wedge \omega_\euc^{n-k}\right)\wedge\left(\pi_3^{23}\right)^*S_3\wedge\left(dd^cv_{1/\lambda}^\Tmc\wedge \Kmc_{v,\eta}^{n-1}\right)\Big]\\
		\notag	&\le\int_{\Xf_{12}} \chi_m^u dd^cu_{1/\lambda}^\Tmc\wedge \Kmc_{u,\eta}^{k-1}\wedge \left(\pi_1^{12}\right)^*S_1\wedge (\pi^{12}_2)^* \left(\pi_2^{23}\right)_*\Big[\left(\chi^j_n\left(\pi_2^{23}\right)^*\phi\right)\\
		\notag	&\quad\quad\quad\quad\quad\quad\quad\quad\quad\quad\left(\pi_2^{23}\right)^*\left( S_2\wedge \omega_\euc^{n-k}\right)\wedge\left(\pi_3^{23}\right)^*S_3\wedge\left(dd^cv_{1/\lambda}^\Tmc\wedge \Kmc_{v,\eta}^{n-1}\right)\Big]	\\
		\label{eq:claim5}	&=\int_{\Xf_{12}}u_{1/\lambda}^\Tmc dd^c \Big[\chi_m^u \Kmc_{u,\eta}^{k-1}\wedge \left(\pi_1^{12}\right)^*S_1\wedge (\pi^{12}_2)^* \left(\pi_2^{23}\right)_*\Big[\left(\chi^n_j\left(\pi_2^{23}\right)^*\phi\right)\\
		\notag	&\quad\quad\quad\quad\quad\quad\quad\quad\quad\quad\left(\pi_2^{23}\right)^*\left( S_2\wedge \omega_\euc^{n-k}\right)\wedge\left(\pi_3^{23}\right)^*S_3\wedge\left(dd^cv_{1/\lambda}^\Tmc\wedge \Kmc_{v,\eta}^{n-1}\right)\Big]\Big]
\end{align}}
All the integrals in \eqref{eq:claim5} other than
\begin{align}
	&\int_{\Xf_{12}}\chi_m^u u_{1/\lambda}^\Tmc \Kmc_{u,\eta}^{k-1}\wedge \left(\pi_1^{12}\right)^*S_1\wedge (\pi^{12}_2)^* \left(\pi_2^{23}\right)_*\Big[\left(\chi^n_j\left(\pi_2^{23}\right)^*(dd^c\phi)\right)\\
	\notag	&\quad\quad\quad\quad\quad\quad\quad\left(\pi_2^{23}\right)^*\left( S_2\wedge \omega_\euc^{n-k}\right)\wedge\left(\pi_3^{23}\right)^*S_3\wedge\left(dd^cv_{1/\lambda}^\Tmc\wedge \Kmc_{v,\eta}^{n-1}\right)\Big]
\end{align} 
contain either a derivative of $\chi_m^v$ or that of $\chi_j^n$. Then, as previously, Condition (I) for $S_1$ and $S_2$ (induced from the continuity assumption on $S_1$) or Condition (I) for $S_2$ and $S_3$ shows the uniform boundedness of the mass independently of $\eta$ and $\lambda$. The integral can be bounded by
\begin{align}
	&\int_{\Xf_{12}}\chi_j^1u_{1/\lambda}^\Tmc \Kmc_{u,\eta}^{k-1}\wedge \left(\pi_1^{12}\right)^*S_1\wedge (\pi^{12}_2)^*\Big[\left( S_2\wedge \omega_\euc^{n-k}\right)\\
	\notag	&\quad\quad\quad\quad\quad\quad\quad\quad\quad\quad\wedge\left(\pi_2^{23}\right)_*\Big[\left(\pi_3^{23}\right)^*S_3\wedge\left(dd^cv_{1/\lambda}^\Tmc\wedge \Kmc_{v,\eta}^{n-1}\right)\wedge \omega_{23}\Big]\Big]\\
	\notag	&=\int_{\Xf_{12}}\chi_j^1u_{1/\lambda}^\Tmc \Kmc_{u,\eta}^{k-1}\wedge \left(\pi_1^{12}\right)^*S_1\wedge (\pi^{12}_2)^*\Big[\left( S_2\wedge \omega_\euc^{n-k}\right)\wedge\left(\Tc_{1/\lambda}\Kmc_{\theta'}\right)^{n+1, n}(S_3)\Big]
\end{align} 
up to a multiplicative constant. By Lemma \ref{lem:DSHness}, it is possible to write $(\Tc_{1/\lambda}\Kc_{\theta'})^{n+1, n}(S_3)$ as the sum of a positive closed current and a smooth form in $U_j^1$ as follows:
\begin{align*}
	[(\Tc_{1/\lambda}\Kc_{\theta'})^{n+1, n}(S_3)-A(\theta_0, 1/\lambda, \theta')] + A(\theta_0, 1/\lambda, \theta'),
\end{align*}
where $A(\theta_0, 1/\lambda, \theta'):=\left(\Tc_{1/\lambda}\Kc_{\theta'}\right)_{\theta_0}^{n+1,n}(S_3)+c'_{TK}\|S_3\||\theta_0|^{-2n}\omega_\euc^{s_3+1}$.

The positive closed current $S_2\wedge \omega_{euc}^{n-k}\wedge [(\Tc_{1/\lambda}\Kc_{\theta'})^{n+1, n}(S_3)-A(\theta_0, 1/\lambda, \theta')]$ in $U_j^1$ has uniformly bounded mass as it converges according to Condition (I) for $S_2$ and $S_3$. The uniform convergence as in Propositions \ref{prop:unif_conv_poten_sup} and \ref{prop:uniform_classical_prod} implies the uniform boundedness of the integral
{\begin{align*}
		&\int_{\Xf_{12}}\chi_j^1u_{1/\lambda}^\Tmc \Kmc_{u,\eta}^{k-1}\wedge \left(\pi_1^{12}\right)^*S_1\wedge (\pi^{12}_2)^*\Big[S_2\wedge \omega_{euc}^{n-k}\wedge [(\Tc_{1/\lambda}\Kc_{\theta'})^{n+1, n}(S_3)-A(\theta_0, 1/\lambda, \theta')]\Big].
\end{align*}}
Since $A(\theta_0, 1/\lambda, \theta')$ is smooth, for the other integral
\begin{align*}
	\int_{\Xf_{12}}\chi_j^1u_{1/\lambda}^\Tmc \Kmc_{u,\eta}^{k-1}\wedge \left(\pi_1^{12}\right)^*S_1\wedge (\pi^{12}_2)^*\Big[S_2\wedge \omega_{euc}^{n-k}\wedge A(\theta_0, 1/\lambda, \theta')\Big],
\end{align*}
the uniform boundedness is clear from Condition (I) for $S_1$ and $S_2$ induced from the continuity assumption on $S_1$.\medskip

For limit currents having no mass on $\Delta$, we assume that $k\le n-1$. By passing to a subsequence, we may assume the sequence converges as $\eta\to 0$ and $\lambda\to\infty$. We consider
\begin{align*}
	&\int_{\Xf_{12}} \chi_m^u dd^cu_{1/\lambda}^\Tmc\wedge \Kmc_{u,\eta}^{k-1}\wedge \left(\pi_1^{12}\right)^*S_1\wedge (\pi^{12}_2)^* \left(\pi_2^{23}\right)_*\Big[\left(\chi^n_j\left(\pi_2^{23}\right)^*\phi\right)\\
	\notag	&\quad\quad\quad\quad\quad\quad\quad\left(\pi_2^{23}\right)^*\left( S_2\wedge \omega_\euc^{n-k}\right)\wedge\left(\pi_3^{23}\right)^*S_3\wedge\left(dd^cv_{1/\lambda}^\Tmc\wedge \Kmc_{v,\eta}^{n-1}\right)\Big].
\end{align*}
Notice that it can be bounded by a linear combination of
\begin{align*}
	&\frac{1}{\log m}\int_{\Xf_{12}} u^\Kmc_{1/m}\wedge dd^cu_{1/\lambda}^\Tmc\wedge \Kmc_{u,\eta}^{k-1}\wedge \left(\pi_1^{12}\right)^*S_1\wedge (\pi^{12}_2)^* \left(\pi_2^{23}\right)_*\Big[\chi_j^1\\
	\notag	&\quad\quad\quad\quad\quad\quad\quad\left(\pi_2^{23}\right)^*\left( S_2\wedge \omega_\euc^{n-k}\right)\wedge\left(\pi_3^{23}\right)^*S_3\wedge\left(dd^cv_{1/\lambda}^\Tmc\wedge \Kmc_{v,\eta}^{n-1}\right)\Big].
\end{align*}
and another integral with $1/m$ replaced by $1/m^2$. Since $k\le n-1$, the same argument as previously implies the integral is uniformly bounded. Then, as $\log m\to\infty$ as $m\to\infty$, we get the desired convergence to $0$.
%

\subsection{Proof of Claim 4}
The shadow of the tangent currents of $T$ equals $\big(S_1\wedge \big(S_2\wedge S_3\big)_{DS}\big)_{DS}$.\medskip

Observe that on $E$, the form $\Kmc_{u, 1}^n$ and $\Kmc_{v, 1}^n$ are smooth $(n, n)$-forms with supports near the analytic subsets $z=0$ and $w=0$, respectively. Hence, $\Kmc_{u, 1}^n\wedge\Kmc_{v, 1}^n$ is a smooth form that has compact support in every fiber of $E$ and therefore, extends to $\overline{E}$. In particular, it is cohomologous to the linear space of $\P^{2n}$ on each fiber. Hence, for a smooth test form $\phi$ on $U_j^n$, the shadow of the tangent current of $T$ is computed as
\begin{align*}
	&\lim_{\lambda\to\infty}(A_\lambda)_*(\pi_1^*S_1\wedge\pi_2^*S_2\wedge\pi_3^*S_3)\wedge \pi_{12}^*\Kmc_{u, 1}^n\wedge\pi_{23}^*\Kmc_{v, 1}^n\wedge \pi_2^*\phi\\
	&=\lim_{\lambda\to\infty}(\pi_1^*S_1\wedge\pi_2^*S_2\wedge\pi_3^*S_3)\wedge \pi_{12}^*\Kmc_{u, 1/\lambda}^n\wedge\pi_{23}^*\Kmc_{v, 1/\lambda}^n\wedge \pi_2^*\phi\\
	&=\lim_{\lambda\to\infty}(\pi_2^*(\phi S_2))\wedge \pi_1^*S_1\wedge\pi_{12}^*\Kmc_{u, 1/\lambda}^n\wedge\pi_3^*S_3\wedge\pi_{23}^*\Kmc_{v, 1/\lambda}^n\\
	&=\lim_{\lambda\to\infty}\phi S_2\wedge (\pi_2)_*\left(\pi_1^*S_1\wedge\pi_{12}^*\Kmc_{u, 1/\lambda}^n\right)\wedge(\pi_2)_*\left(\pi_3^*S_3\wedge\pi_{23}^*\Kmc_{v, 1/\lambda}^n\right)\\
	&=\lim_{\lambda\to\infty}\phi S_2\wedge \left(\pi_2^{12}\right)_*\left(\left(\pi_1^{12}\right)^*S_1\wedge\Kmc_{u, 1/\lambda}^n\right)\wedge\left(\pi_2^{23}\right)_*\left(\left(\pi_3^{23}\right)^*S_3\wedge\Kmc_{v, 1/\lambda}^n\right)\\
	&=\lim_{\lambda\to\infty}\phi S_2\wedge \left(\pi_2^{12}\right)_*\left(\left(\pi_1^{12}\right)^*S_1\wedge\Kmc_{u, 1/\lambda}^n\right)\wedge\left(\pi_2^{23}\right)_*\left(\left(\pi_3^{23}\right)^*S_3\wedge\Kmc_{v, 1/\lambda}^n\right)\\
	&=\lim_{\lambda\to\infty}\chi_j^n\left(\pi_1^{12}\right)^*S_1\wedge\Kmc_{u, 1/\lambda}^n\wedge \left(\pi_2^{12}\right)^*\Big[\phi S_2\wedge \left(\pi_2^{23}\right)_*\left(\left(\pi_3^{23}\right)^*S_3\wedge\Kmc_{v, 1/\lambda}^n\right)\Big]\\
	&=\lim_{\lambda\to\infty}u^\Kmc_{1/\lambda}dd^c\Big[\chi_j^n(dd^cu)^{n-1}\wedge\left(\pi_1^{12}\right)^*S_1\wedge \left(\pi_2^{12}\right)^*\Big[\phi S_2\wedge \left(\pi_2^{23}\right)_*\left(\left(\pi_3^{23}\right)^*S_3\wedge\Kmc_{v, 1/\lambda}^n\right)\Big]\Big].
\end{align*} 

We consider the following double limit:
\begin{align*}
	&\lim_{\substack{\lambda\to\infty\\t\to\infty}}u^\Kmc_{1/\lambda}dd^c\Big[\chi_j^n(dd^cu)^{n-1}\wedge\left(\pi_1^{12}\right)^*S_1\wedge \left(\pi_2^{12}\right)^*\Big[\phi S_2\wedge \left(\pi_2^{23}\right)_*\left(\left(\pi_3^{23}\right)^*S_3\wedge\Kmc_{v, 1/t}^n\right)\Big]\Big]
\end{align*}
All the integrals in the above other than
{\begin{align*}
		\chi_j^nu^\Kmc_{1/\lambda}(dd^cu)^{n-1}\wedge\left(\pi_1^{12}\right)^*S_1\wedge \left(\pi_2^{12}\right)^*\Big[dd^c\phi\wedge S_2\wedge \left(\pi_2^{23}\right)_*\left(\left(\pi_3^{23}\right)^*S_3\wedge\Kmc_{v, 1/t}^n\right)\Big]
\end{align*}} 
contain either a derivative of $\chi_j^n$. Then, as previously, Condition (I) for $S_2$ and $S_3$ shows the uniform convergence as $\lambda\to\infty$ with respect to $t$. In the above integral, the current $dd^c\phi\wedge S_2\wedge \left(\pi_2^{23}\right)_*\left(\left(\pi_3^{23}\right)^*S_3\wedge\Kmc_{v, 1/\lambda}^n\right)$ is of uniformly bounded $*$-norm on $U_j^0$. Proposition \ref{prop:unif_conv_poten_sup} implies the uniform convergence with respect to $t$. Hence, the diagonal limit exists and is the same as iterated limit. \hfill $\qed$ 

\section{Positive closed currents with continuous superpotentials on compact K\"ahler manifolds}\label{sec:cptK}
We consider a compact K\"ahler manifold $(X, \omega_X)$ of dimension $n$. In this section, $\Cc_p(X)$ denotes the set of positive closed $(p, p)$-currents of unit mass and $\widetilde{\Cc_p(X)}$ denotes its subset of smooth ones. We take $\left(\left(U_j^i, \xi_j^i\right)\right)_{i=0, 1, \ldots, n, j\in J}$ as localizing data for $X$. We may assume that $J$ is finite. We use the following class of currents: $\widetilde{K}:=\widetilde{\Cc_{n-p+1}(X)}$ and $K:=\overline{\widetilde{\Cc_{n-p+1}(X)}}$.\medskip

The key point is that the continuity of superpotentials can be described in terms of integration as in \cite{DNV} and localization works well with integration.\medskip

The next theorem proves that continuous superpotentials are a special case of continuous local potential functionals.
\begin{theorem}
	Let $(X, \omega_X)$ be a compact K\"ahler manifold of dimension $n$.  Let $S\in\Cc_p(X)$. Then, $S$ admits continuous superpotentials if and only if $S$ admits continuous local potential functionals on $K$, where $K$ is the closure of the set of smooth positive closed $(n-p+1, n-p+1)$-currents of unit mass.
\end{theorem}

\begin{proof}
	We recall the context for \cite[Proposition 2.7]{DNV}. Let $\Xf=X^2$ and $\pi_k:\Xf\to X$ the canonical projection onto the $k$-th factor for $k=1, 2$. We denote by $\Delta$ the diagonal submanifold of $\Xf$. Let $\pi:\Xfh\to \Xf$ be the blow-up of $\Xf$ along $\Delta$. Let $\Pi_k: \Xfh\to X_k$ be defined by $\Pi_k:=\pi_k\circ \pi$ for $k=1, 2$. Note that $\Pi_1$ and $\Pi_2$ are submersions. Let $\widehat{\Delta}:=\pi^{-1}(\Delta)$ denote the exceptional divisor. Choose a real smooth closed $(1, 1)$-form $\widehat{\beta}$ on $\Xfh$ which is cohomologous to $[\widehat{\Delta}]$. Then, we can write $[\widehat{\Delta}]-\widehat{\beta}=dd^c\widehat{u}$. This equation implies that $\widehat{u}$ is smooth outside $\widehat{\Delta}$ and $\widehat{u}-\log \dist(\cdot, \widehat{\Delta})$ is a bounded function near $\widehat{\Delta}$. After subtraction of a proper constant, we may assume that $\widehat{u}$ is negative. By Blanchard's theorem, $\Xfh$ is a K\"ahler manifold and we fix a K\"ahler form $\widehat{\Omega}$. Recall the following function in \cite{DNV}.
	\begin{align*}
		\nu^X_S(\delta):=\sup_{R\in\widetilde{K}}\int_{\widehat{\Delta}_\delta}-\widehat{u}\widehat{\Omega}^{n-1}\wedge\Pi_1^*(S)\wedge \Pi_2^*(R). 
	\end{align*}
	By \cite[Proposition 2.7]{DNV}, $S$ admits continuous superpotentials if and only if $\lim_\delta\nu^X_S(\delta)=0$. As the continuity of superpotentials is described in terms of integrals, we can locally analyze it. Namely, we compare the integral $\int_{\widehat{\Delta}_\delta}-\widehat{u}\widehat{\Omega}^{n-1}\wedge\Pi_1^*(S)\wedge \Pi_2^*(R)$ with $\Fc_{S, j}^i(R)$ for $i=1, \ldots, n$ on $U_j^0$ for $j\in J$.\medskip
	
	Note that since $R$ is assumed to be smooth, the integrand in $-\widehat{u}\widehat{\Omega}^{n-1}\wedge\Pi_1^*(S)\wedge \Pi_2^*(R)$ and in $\Fc_{S, j}^i(R)$ for $i=1, \ldots, n$ and $j\in J$ has no mass on the exceptional divisor and on the diagonal submanifold, respectively.\medskip
	
	We compute in local coordinates. Let $(x, y)\in \C^n\times \C^n$. Let $(\C^n\times \C^n)_\bl$ be the blow-up of $\C^n\times \C^n$ along the diagonal submanifold $\Delta:=\{(x, x)\in \C^n\times \C^n\}$. The blow-up is a $n$-dimensional submanifold of $\C^{2n}\times \P^{n-1}$ and we let $(x, y, [w])\in \C^n\times \C^n\times \P^{n-1}$ denote the coordinates. Here, for $w\in\C^n\setminus \{0\}$, $[w]$ denotes the line $[w_1: \cdots: w_n]$ passing through the origin and $w$. The form $dd^c|x-y|^2 + dd^c\log|w|^2$ defines a natural K\"ahler form on $\C^n\times \C^n\times \P^{n-1}$.\medskip
	
	The blow-up $(\C^n\times \C^n)_\bl$ is the closure of the set $\{(x, y, [x-y])\in \C^n\times \C^n\times \P^{n-1}: (x, y)\in \C^n\times \C^n\setminus \Delta\}$. So, the restriction of $dd^c\left(|x|^2 +|x-y|^2\right)+dd^c\log|w|$ to $(\C^n\times\C^n)_\bl$ defines a natural K\"ahler form on $(\C^n\times\C^n)_\bl$, which is equivalent to $\widehat{\Omega}$. Outside the exceptional set $\widehat{\Omega}$, as $(\C^n\times \C^n)_\bl$ is biholomorphic to $\C^n\times\C^n$, it can be written as $\omega_\Xfh:=dd^c\left(|x|^2 +|x-y|^2\right)+dd^c\log|x-y|$ in terms of the coordiantes of $(x, y)\in \C^n\times \C^n$.\medskip
	
	We are actually interested in a deleted neighborhood of the diagonal submanifold as each $U_j^0$ is bounded. Without loss of generality, we may consider a sector $\{(x, y)\in \C^n\times\C^n, 0<|x_1-y_1|<1, |x_i-y_i|<2|x_1-y_1| \textrm{ for }i=2, \ldots, n\}$. For other sectors, we apply the same. We change coordinates
	\begin{align*}
		(x, v)\to (x, x_1+v_1, x_2+v_1v_2, \ldots, x_n+v_1v_n, [v])
	\end{align*}
	and use $(x, v)$, where $|v_1|<1$ and $|v_i|<2$ for $i=2, \ldots, n$. With respect to $(x, v)$, we have
	{\begin{align*}
			\omega_\Xfh=dd^c|x|^2+\left(1+\sum_{j=2}^n\left|v_j\right|^2\right)dd^c|v_1|^2+\left|v_1\right|^2\left(\sum_{j=2}^n dd^c\left|v_j\right|^2\right)+dd^c\log\left(1+\sum_{j=2}^n|v_j|^2\right)
	\end{align*}}
	On the other hand, $\pi$ defines a biholomorphism outside $\widehat{\Delta}$, with respect to the coordinates $(x, v)$, $\omega:=\pi_1^*\omega_\euc+\pi_1^*\omega_\euc$ becomes
	\begin{align*}
		\pi^*\omega=dd^c|x|^2+\left(1+\sum_{j=2}^n\left|v_j\right|^2\right)dd^c|v_1|^2+\left|v_1\right|^2\left(\sum_{j=2}^n dd^c\left|v_j\right|^2\right)
	\end{align*}
	and the form $dd^cu$ becomes
	\begin{align*}
		\pi^*\left(dd^cu\right)=dd^c\log\left(\left|v_1\right|^2\left(1+\sum_{j=2}^n\left|v_j\right|^2\right)\right)=dd^c\log\left(1+\sum_{j=2}^n\left|v_j\right|^2\right).
	\end{align*}
	Thanks to the positivity, we see that there exists a constant $c_1, c_2>1$ such that over the sector, we have
	\begin{align}\label{eq:equivalence_between_norms}
		\widehat{\Omega}^{n-1}\le c_1\omega_\Xfh^{n-1}\le c_2\pi^*\left(\sum_{k=0}^{n-1}\left(dd^cu\right)^k\wedge\omega^{n-k-1}\right).
	\end{align}
	Also, it is not difficult to see that in this region, there exists a constant $c>1$ such that
	\begin{align*}
		c^{-1}\dist(\pi(\cdot), \Delta)\le \dist(\cdot, \widehat{\Delta})\le c\dist(\pi(\cdot), \Delta).
	\end{align*}
	Indeed, we have
	\begin{align*}
		|v_1|^2\le |x-y|^2 = |v_1|^2\left(1+\sum_{i=2}^n|v_i|^2\right)\le 4n|v_1|^2 
	\end{align*}
	Due to the negativity and positivity of each current in the integrand of the integral and $\Fc_{S, j}^i(R)$ for $i=1, \ldots, n$ and $j\in J$, we see that the integral $\int_{\widehat{\Delta}_\delta}-\widehat{u}\widehat{\Omega}^{n-1}\wedge\Pi_1^*(S)\wedge \Pi_2^*(R)$ is equivalent to the sum of the functions $\Fc_{S, j}^i(R)$ for $i=1, \ldots, n$ on $U_j^0$ for each $j\in J$ and the integral dominates every $\Fc_{S, j}^i(R)$. This completes the proof.
\end{proof}


We consider the associativity of the Dinh-Sibony product on compact K\"ahler manifolds. 

\begin{theorem}
	Let $(X, \omega_X)$ be a compact K\"ahler manifold of dimension $n$. Let $S_1$, $S_2$ and $S_3$ be positive closed currents of bidegree $(s_1, s_1)$, $(s_2, s_2)$ and $(s_3, s_3)$, respectively. Suppose that $S_1$ admits continuous superpotential and $S_2$ and $S_3$ satisfy Condition (I). Then, the Dinh-Sibony product of $\big(S_1\wedge S_2\big)_{DS}$ and $S_3$, and the Dinh-Sibony product of $S_1$, $S_2$ and $S_3$ are well defined for each $j\in J$, and we have
	\begin{align*}
		\left(S_1\wedge (S_2\wedge S_3)_{DS}\right)_{DS}=\left((S_1\wedge S_2)_{DS}\wedge S_3\right)_{DS}=(S_1\wedge S_2\wedge S_3)_{DS}\textrm{ on }X.
	\end{align*} 
\end{theorem}

For the proof, we introduce some more notations.
\begin{center}
	\begin{tikzpicture}
		\node		(buX12)				at (-3, 0) {$\widehat{\Xf_{12}}$};
		\node		(buX23)				at (3, 0) {$\widehat{\Xf_{23}}$};
		\node 		(X123) 				at (0, 0) {$X^3:=X_1\times X_2\times X_3$};
		\node      	(X23)               at (1.5, -1.5) 			{$\Xf_{23}:=X_2\times X_3$};
		\node       (X12)       		at (-1.5, -1.5) 			{$\Xf_{12}:=X_1\times X_2$};
		\node       (X1)      			at (-3, -3) 			{$X_1$};
		\node 		(X2)				at (0, -3)			{$X_2$};
		\node		(X3)				at (3, -3)			{$X_3$};
		\draw[->] (X123) -- (X12) node[midway, left] {$\pi_{12}$};
		\draw[->] (X123) -- (X23) node[midway, right] {$\pi_{23}$};
		\draw[->] (X12) -- (X1) node[midway, left] {$\pi^{12}_1$};
		\draw[->] (X23) -- (X2) node[midway] {$\pi^{23}_2\,\,\,\,\,\,\,\,$};
		\draw[->] (X12) -- (X2) node[midway, left] {$\pi^{12}_2$};
		\draw[->] (X23) -- (X3) node[midway, right] {$\pi^{23}_3$};
		\draw[->] (buX12) -- (X12) node[midway, left] {$\pi_{12}^\Delta$};
		\draw[->] (buX23) -- (X23) node[midway, right] {$\pi_{23}^\Delta$};
	\end{tikzpicture}
\end{center}
Let $\Delta_{12}$, $\Delta_{23}$ denote the diagonal submanifolds of $\Xf_{12}$, $\Xf_{23}$, respectively. Let $\pi_{12}^\Delta:\widehat{\Xf_{12}}\to \Xf_{12}$, $\pi_{23}^\Delta:\widehat{\Xf_{23}}\to \Xf_{23}$ denote the blow-ups of $\Xf_{12}$, $\Xf_{23}$ along $\Delta_{12}$, $\Delta_{23}$, respectively. We denote by $\widehat{\Delta_{12}}:=\left(\pi_{12}^\Delta\right)^{-1}(\Delta)$, $\widehat{\Delta_{23}}:=\left(\pi_{23}^\Delta\right)^{-1}(\Delta)$. We let $\alpha_{12}$, $\alpha_{23}$ be smooth closed $(1, 1)$-currents cohomologous to $[\widehat{\Delta_{12}}]$, $[\widehat{\Delta_{23}}]$, and $\hat u$, $\hat v$ q-psh functions such that $[\widehat{\Delta_{12}}]-\alpha_{12}=dd^c \hat u$, $[\widehat{\Delta_{23}}]-\alpha_{23}=dd^c \hat v$, respectively. Let $\widehat{\omega_{12}}$ and $\widehat{\omega_{23}}$ be K\"ahler form like $\widehat{\Omega}$ in the above on $\Xf_{12}$ and $\Xf_{23}$, respectively. For notational convenience, we write $\Pi_1^{12}:=\pi_1^{12}\circ \pi_{12}^\Delta$,  $\Pi_1^{12}:=\pi_2^{12}\circ \pi_{12}^\Delta$, $\Pi_2^{23}:=\pi_2^{23}\circ \pi_{23}^\Delta$,  $\Pi_3^{23}:=\pi_2^{23}\circ \pi_{23}^\Delta$. For $\theta\in \C^*$ with $|\theta|\ll 1$, we let $\hat u_\theta:=\chi({\hat u}-\log|\theta|)+\log|\theta|$ and $\hat v_\theta:=\chi({\hat v}-\log|\theta|)+\log|\theta|$.\medskip

As the proof goes in the same way as the proof of Theorem \ref{thm:associativity}, we rather give its sketch for a concise presentation.
\begin{proof}[Sketch of Proof]
	We explain the modification of the proof of Theorem \ref{thm:associativity}. On compact K\"ahler manifolds, the existence of tangent currents are guaranteed. Hence, different from the case of Theorem \ref{thm:associativity}, we only care about the $h$-dimension and the unique shadow of tangent currents. So, instead of the kernels of the form $u_\theta^\Tmc\wedge \Kmc_{\theta'}^{i-1}$, we only need to use $\Kmc_\theta^i$, which possibly makes computations simper. For example, Claim 1 changes to
	\begin{align*}
		J_i(k, \theta):=\int_{\Xf_{23}}\varphi \chi_k^\nu K_\theta^i\wedge \left(\pi_2^{23}\right)^*(S_1\wedge S_2)_{DS}\wedge \left(\pi_3^{23}\right)^*S_3\wedge \omega_{23}^{n-i}\to 0
	\end{align*}
	as $\theta\to 0$ and $k\to\infty$ for $i=0, \ldots, n-1$.\medskip
	
	We describe how the lemmas and propositions are to be modified. We replace $u^\Kmc_{(\cdot)}$, $v^\Kmc_{(\cdot)}$ in the definition of $\chi_k^u$, $\chi_k^v$ in Lemmas \ref{lem:cut_off_12} and \ref{lem:cut_off_23} by $\left(\pi_{12}^\Delta\right)_*{\hat u}_{(\cdot)}$, $\left(\pi_{12}^\Delta\right)_*{\hat v}_{(\cdot)}$, respectively. Note that $u=\log|x-y|$, $\left(\pi_{12}^\Delta\right)^*u-{\hat u}$ is a smooth function in a neighborhood of $\widehat{\Delta_{12}}\cap \left(\pi_{12}^\Delta\right)^{-1}\left(U_j^0\times U_j^0\right)$. Hence, for some constant $c_\pi>0$, we have
	\begin{align}
		\label{eq:comparison_poten_hypersurf}dd^c{\hat u}_\theta - c_\pi\widehat{\omega_{12}}\le dd^c\left(\pi_{23}^\Delta\right)^*u_\theta^\Kmc\le dd^c{\hat u}_\theta + c_\pi\widehat{\omega_{12}}
	\end{align}
	in a neighborhood of $\widehat{\Delta_{23}}$ on $\left(\pi_{23}^\Delta\right)^{-1}\left(U_j^0\times U_j^0\right)$. Note also that there exists a constant $c_\reg>0$ such that $dd^c{\hat u}_\theta \ge -c_\reg\widehat{\omega_{12}}$ for all $\theta$ with $|\theta|\ll 1$ as in \cite[Lemma 2.4.2]{DS10}. The same is true for $v$ and related functions. The main ingredients for the proof of Claim 1 through 4 in Theorem \ref{thm:associativity} are the limits of the following types:
	\begin{enumerate}
		\item for $i=0, \ldots, n-1$ and for $j=0, 1, \ldots, n$,
		\begin{align*}
			&\int_{\left(\Delta_{12}\right)_\varepsilon\setminus \Delta_{12}}u(dd^cu)^i\wedge\omega_{12}^{n-i-1}\wedge\left(\pi_1^{12}\right)^*S_1\wedge \left(\pi_2^{12}\right)^*\Big[S_2\\
			\notag&\quad\quad\quad\quad\quad\quad\quad\quad\quad\quad\quad\quad\quad\quad\quad\wedge \left(\pi_2^{23}\right)_*\Big[\left(\pi_3^{23}\right)^*S_3\wedge \Kmc_{v, \theta}^j\wedge \omega_{23}^{n-j+1}\Big]\Big]\to 0
		\end{align*}
		as $\theta \to 0$ and then $\varepsilon\to 0$. Here, by convention $\Kmc^0_{v, \theta}=1$.
		\item for $i=0, \ldots, n-1$ and for $j=0, 1, \ldots, n$,
		\begin{align*}
			&\int_{\Xf_{12}}u(dd^cu)^i\wedge\omega_{12}^{n-i-1}\wedge\left(\pi_1^{12}\right)^*S_1\wedge \left(\pi_2^{12}\right)^*\Big[ S_2\\
			\notag&\quad\quad\quad\quad\quad\quad\quad\quad\quad\quad\quad\quad\quad\quad\wedge \left(\pi_2^{23}\right)_*\Big[dd^c\left(\chi_k^\nu\right)^2\wedge\left(\pi_3^{23}\right)^*S_3\wedge \Kmc_{v, \theta}^j\wedge \omega_{23}^{n-j}\Big]\Big]\to 0
		\end{align*}
		as $\theta \to 0$ and then $k\to \infty$.
	\end{enumerate}
	For these estimates, we use the continuity property for $S_1$. For now, we assume that $S_2\in\Cc_{s_2}(X)$ and $S_3\in \Cc_{s_3}(X)$ be smooth in fixed coholomogy classes. Then, for all $\theta\in\C^*$ with $|\theta|\ll 1$, the currents
	\begin{align*}
		&\Big[S_2\wedge \left(\Pi_2^{23}\right)_*\Big[\left(\Pi_3^{23}\right)^*S_3\wedge (dd^c{\hat v_\theta}+c_\reg\widehat{\Delta_{23}})\wedge\widehat{\omega_{23}}^{n}\Big]\Big],\\
		&\Big[S_2\wedge \left(\Pi_2^{23}\right)_*\Big[dd^c\left(\chi_k^\nu\right)^2\wedge\left(\Pi_3^{23}\right)^*S_3\wedge (dd^c{\hat v_\theta}+c_\reg\widehat{\Delta_{23}})\wedge\widehat{\omega_{23}}^{n-1}\Big]\Big]\quad\textrm{ and }\\
		&\Big[S_2\wedge \left(\Pi_2^{23}\right)_*\Big[\left(\Pi_3^{23}\right)^*S_3\wedge\widehat{\omega_{23}}^{n+1}\Big]\Big]
	\end{align*}
	can be represented by $C^1$-positive closed forms. As in \cite{DS04}, they can be regularized by smooth positive closed forms with respect to the uniform norm. Since $S_1$ admits continuous superpotenitals, \cite[Proposition 2.7]{DNV} implies that there exists a function $M(\varepsilon)>0$ such that $\displaystyle \lim_{\varepsilon\to 0}M(\varepsilon)=0$ and
	\begin{align}
		\label{eq:cpt_compare1}0>&\int_{(\widehat{\Delta_{12}})_\varepsilon}{\hat u}\widehat{\omega_{12}}^{n-1}\wedge\left(\Pi_1^{12}\right)^*S_1\wedge \left(\Pi_2^{12}\right)^*\Big[S_2\\
		\notag&\quad\quad\quad\wedge \left(\Pi_2^{23}\right)_*\Big[\left(\Pi_3^{23}\right)^*S_3\wedge (dd^c{\hat v}_\theta+c_\reg\widehat{\Delta_{23}})\wedge\widehat{\omega_{23}}^{n}\Big]\Big]>-M(\varepsilon)
	\end{align}
	\begin{align}
		\label{eq:cpt_compare2}0>&\int_{(\widehat{\Delta_{12}})_\varepsilon}{\hat u}\widehat{\omega_{12}}^{n-1}\wedge\left(\Pi_1^{12}\right)^*S_1\wedge \left(\Pi_2^{12}\right)^*\Big[S_2\wedge \left(\Pi_2^{23}\right)_*\Big[dd^c\left(\chi_k^\nu\right)^2\\
		\notag&\quad\quad\quad\quad\quad\wedge\left(\Pi_3^{23}\right)^*S_3\wedge (dd^c{\hat v_\theta}+c_\reg\widehat{\Delta_{23}})\wedge\widehat{\omega_{23}}^{n-1}\Big]\Big]>-M(\varepsilon)
	\end{align}
	\noindent and
	\begin{align}
		\label{eq:cpt_compare3}0>&\int_{(\widehat{\Delta_{12}})_\varepsilon}{\hat u}\widehat{\omega_{12}}^{n-1}\wedge\left(\Pi_1^{12}\right)^*S_1\wedge \left(\Pi_2^{12}\right)^*\Big[S_2\wedge \left(\Pi_2^{23}\right)_*\Big[\left(\Pi_3^{23}\right)^*S_3\wedge\widehat{\omega_{23}}^{n+1}\Big]\Big]>-M(\varepsilon)
	\end{align}
	hold for all sufficiently small $\varepsilon>0$. The function $M(\varepsilon)$ may depend on the cohomology classes of $S_2$ and $S_3$.\medskip
	
	We explain how (1) can be deduced from \eqref{eq:cpt_compare1} and \eqref{eq:cpt_compare3}. One can do the same to (2) together with \eqref{eq:cpt_compare2} and \eqref{eq:cpt_compare3}.\medskip
	
	To this end, we extend the class $K$ to a locally defined class $K_{\loc, j}$ on $U_j^0$ by adding some locally defined positive closed currents and we show that local continuous potential functions of $S_1$ are continuous on $K_{\loc, j}$.\smallskip
	
	\noindent Let $m\in\{n, n+1\}$, $i\in\{1, \ldots, n\}$ and $\theta_0\in \C^*$ with $|\theta_0|\ll 1$. For $\theta\in\Cc^*$, we define the transform $\Kc_\theta^{m, i}\left(S_3\right):=\left(\pi_2^{23}\right)_*\Big[\left(\pi_3^{23}\right)^*S_3\wedge \Kmc_{v, \theta}^i\wedge \omega_{23}^{m-i}\Big]$ for $S_3\in \Cc_{s_3}(X)$. As in Lemma \ref{lem:DSHness}, we can find a smooth form $B^{m, i}(\theta_0, S_3)$ such that $\Kc_\theta^{m, i}(S_3)+B^{m, i}(\theta_0, S_3)$ is a smooth current in $\Cc_{s_3}(U_j^1; U_j^0)$ for all $\theta$ with $|\theta|<|\theta_0|$ and that $\|B^{m, i}(\theta_0, S_3)\|_\infty\le \|S_3\||\theta_0|^{-2i}$. 
	We define $K_{\loc, j}\subset \Cc_{s_2+s_3+1}(U_j^1; U_j^0)$ to be the closure of $\widetilde{\Cc_{s_2+s_3+1}(X)}\cup \Big\{S_2\wedge \Big(\Kc_\theta^{n+1, i}(S_3)+B^{n+1, i}(\theta_0, S_3)\Big): S_2\in \widetilde{\Cc_{s_2}(X)}, S_3\in \widetilde{\Cc_{s_3}(X)}, i=0, 1, \ldots, n \textrm{ and }\theta\in \C^* \textrm{ with }|\theta|<|\theta_0|\Big\}$. If we consider $S_2$ and $S_3$ with uniformly bounded mass, then $K_{\loc, j}$ also has uniformly bounded mass.
	\medskip	
	
	Now, we deduce (1) out of \eqref{eq:cpt_compare1} and \eqref{eq:cpt_compare3}. Since $S_1$ admits continuous superpotentials and $S_2$ and $S_3$ are assumed to be smooth, the integrands have no mass on the sets $\widehat {\Delta_{12}}$ and $\widehat{\Delta_{23}}$. Hence, there exists a constant $c_\Delta>0$ independent of $\varepsilon>0$ such that
	\begin{align*}
		&\int_{(\widehat{\Delta_{12}})_{c_\Delta\varepsilon}}{\hat u}\widehat{\omega_{12}}^{n-1}\wedge\left(\Pi_1^{12}\right)^*S_1\wedge \left(\Pi_2^{12}\right)^*\Big[S_2\wedge \left(\Pi_2^{23}\right)_*\Big[\left(\Pi_3^{23}\right)^*S_3\wedge (dd^c{\hat v}_\theta+c_\reg\widehat{\Delta_{23}})\wedge\widehat{\omega_{23}}^{n}\Big]\Big]\\
		&\le\int_{\left(\Delta_{12}\right)_\varepsilon\setminus \Delta_{12}}\chi_j^n\left(\pi_{12}^\Delta\right)_*{\hat u}\left(\left(\pi_{12}^\Delta\right)_*\widehat{\omega_{12}}\right)^{n-1}\wedge\left(\pi_1^{12}\right)^*S_1\wedge \left(\pi_2^{12}\right)^*\Big[S_2\\
		&\quad\quad\quad\quad\quad\quad\quad\quad\quad\quad\wedge \left(\pi_2^{23}\right)_*\Big[\chi_j^n\left(\pi_3^{23}\right)^*S_3\wedge \left(\pi_{23}^\Delta\right)_*(dd^c{\hat v}_\theta+c_\reg\widehat{\Delta_{23}})\wedge\left(\left(\pi_{23}^\Delta\right)_*\widehat{\omega_{23}}\right)^{n}\Big]\Big].
	\end{align*}
	The same is true for \eqref{eq:cpt_compare3}.  Together with \eqref{eq:comparison_poten_hypersurf} and the estimates for \eqref{eq:cpt_compare3}, we utilize the relationship near \eqref{eq:equivalence_between_norms} in the proof of Theorem \ref{thm:cptK}. Note that we need the other direction but it can be done in a similar way. Then, we obtain estimates for
	\begin{align}
		\label{eq:first_local_estimate}&\int_{\left(\Delta_{12}\right)_\varepsilon\setminus \Delta_{12}}\chi_j^nu(dd^cu)^i\wedge\omega_{12}^{n-i-1}\wedge\left(\pi_1^{12}\right)^*S_1\wedge \left(\pi_2^{12}\right)^*\Big[S_2\\
		\notag&\quad\quad\quad\quad\quad\quad\quad\quad\quad\quad\quad\quad\wedge \left(\pi_2^{23}\right)_*\Big[\left(\pi_3^{23}\right)^*S_3\wedge (dd^cv_\theta^\Kmc)\wedge(dd^c v)^j\wedge \omega_{23}^{n-j}\Big]\Big]
	\end{align}
	when $0\le i, j\le n-1$. When, the mass of smooth currents $S_2$ and $S_3$ are uniformly bounded and $0<|\theta|\ll 1$, the integral uniformly converges to $0$ with respect to $S_2$, $S_3$, $\theta$ as $\varepsilon\to 0$ as desired. In particular, we can deduce that the local potential functions are continuous on $K_{\loc, j}$, which is also used in the proof of Claim 1 through Claim 4.\medskip
	
	For general positive closed currents $S_2\in \Cc_{s_2}(X)$ and $S_3\in \Cc_{s_3}(X)$. According to \cite{DS04}, we can write $S_2=S_2^+-S_2^-$ and $S_3=S_3^+-S_3^-$, where all $S_2^\pm\in \Cc_{s_2}(X)$, $S_3^\pm\in \Cc_{s_3}(X)$ can be approximated by currents in $\widetilde{\Cc_{s_2}(X)}$, $\widetilde{\Cc_{s_3}(X)}$, respectively. The aforementioned continuity for $S_2^\pm\in \Cc_{s_2}(X)$, $S_3^\pm\in \Cc_{s_3}(X)$, we obtain the estimates for general $S_2\in \Cc_{s_2}(X)$ and $S_3\in \Cc_{s_3}(X)$ and related continuity properties.
	\medskip
	
	In a similar way, we can apply the above arguments to (2) with \eqref{eq:cpt_compare2} and \eqref{eq:cpt_compare3}. With (1), (2) and related continuity properties, we can modify all the proofs of Claim 1 through Claim 4.
\end{proof}

\section{The continuity of the Dinh-Sibony product}\label{sec:conti_DS}
Applying the same method as in the proof of Propositions \ref{prop:char_conti}, \ref{prop:unif_conv_poten_sup} and \ref{prop:uniform_classical_prod}, we obtain a sufficient condition on the continuity of the Dinh-Sibony product as follows:                   

\begin{theorem}\label{thm:coni_DS_prod}
	Let $X$ be a complex manifold of dimension $n$. Let $S$ be a positive closed $(p, p)$-current on $X$. Let $\left(\left(U_j^i, \xi_j^i\right)\right)_{i=0, 1, \ldots, n, j\in J}$ be localizing data. Let $R$, $(R_k)_{k\in \N}$ be positive closed $(r, r)$-currents on $X$ satisfying Condition (I) with $S$ on $X$, where $1\le r\le n-s$. For each $j\in J$ and $0<\varepsilon_0\ll 1$, we define the function $\mu_j:(0, \varepsilon_0)\to \R_{\ge 0}$ by
	\begin{align*}
		\mu_j(\varepsilon)=\sup_{k\in\N}\int_{\Delta_\varepsilon\setminus \Delta}-\chi_j^nu(dd^cu)^{n-1}\pi_1^*S\wedge\pi_2^*R_k \wedge \omega^{n-s-r+1}
	\end{align*}
	Suppose that $\displaystyle \lim_{k\to\infty}R_k=R$ and for each $j\in J$, we have $\displaystyle \lim_{\varepsilon\to 0}\mu_j(\varepsilon)=0$. Then, we have $(S\wedge R_k)_{DS}\to (S\wedge R)_{DS}$ in the sense of currents as $k\to \infty$.
\end{theorem}

\bibliographystyle{plain}
\bibliography{intersection}

\end{document}